\documentclass{amsart}
\usepackage{graphicx,color,tikz}
\usepackage{fullpage}
\usepackage{listings}
\lstdefinelanguage{maple}
	{morekeywords={true, false, try, catch, return, break, error, 
	               module, export, local, option, in, use,
                 and, or, not, xor, xnor,
                 if, then, elif, else, fi,
                 while, for, from, by, to, do, od,
                 proc, nargs, local, global, end, NULL}}
\renewcommand\r{\mathbf{r}}
\newcommand\s{\mathbf{s}}
\newcommand\vone{\mathbf{1}}

\usetikzlibrary{positioning}
\usetikzlibrary{decorations.pathreplacing}

\tikzset{partition/.style={fill,circle,inner sep=1pt},
         part/.style={baseline=0,scale=0.5,bend left=45},
         partlabel/.style={below}}
\usetikzlibrary{shapes,arrows}
\tikzstyle{pnt}=[draw,ellipse,fill,inner sep=1pt]

\tikzstyle{opnt}=[ ]
\tikzstyle{pntt}=[draw,ellipse,fill,inner sep=0.5pt]
\tikzstyle{point}=[draw,ellipse,fill,inner sep=2pt]

\newtheorem{theorem}{Theorem}[section]

\newtheorem{prop}[theorem]{Proposition}

\newtheorem{corollary}[theorem]{Corollary}
\newtheorem{conjecture}[theorem]{Conjecture}
\newcommand\bt{\begin{tabular}}
\newcommand\et{\end{tabular}}
\newtheorem{defn}{Definition}
\newtheorem{claim}{Claim}
\newtheorem*{example}{Example}

\newcommand{\Succ}{\operatorname{succ}}

\newcommand{\diagpi}{ \foreach \i in {1,...,9}
        \node[pnt,label=below:$\i$] at (\i,0)(\i) {};
   \node[opnt] at (10,1)(1a){};
   \node[opnt] at (10,1.5)(2a){};
   \node[opnt] at (10,2)(3a){};
   \draw(1)  to [bend left=45] (3);
   \draw(3)  to [bend left=45] (3a);
   \draw(4)  to [bend left=45] (6);
   \draw(5)  to [bend left=45] (2a);
   \draw(7)  to [bend left=45] (1a);
   \draw(8)  to [bend left=45] (9);}

\title[$k$-nonnesting partitions and permutations]{A generating tree approach to\\ $k$-nonnesting partitions and permutations}
\author{Sophie Burrill}
\author{Sergi Elizalde}
\address[S.~Elizalde]{Department of Mathematics, Dartmouth College, Hanover, NH 03755, USA}
\email{sergi.elizade@dartmouth.edu}
\author{Marni Mishna}
\author{Lily Yen}
\address[S.~Burrill, M.~Mishna, L.~Yen]{Department of Mathematics, Simon Fraser University, Burnaby, BC, Canada}
\address[L.~Yen]{Department of Mathematics and Statistics, Capilano University, North Vancouver, BC, Canada} 
\thanks{The authors are grateful to the NSERC Discovery Grant Program and the Pacific Institute of
  Mathematical Sciences (Canada) for facilitating this collaboration. The second author was also partially supported by grant DMS-1001046 from the NSF and by grant \#280575 from the Simons Foundation.
}

\begin{document}
\begin{abstract}
We describe a generating tree approach to the enumeration and exhaustive
generation of $k$-nonnesting set partitions and permutations. Unlike previous work in the literature
using the connections of these objects to Young tableaux and restricted lattice walks, our
approach deals directly with partition and permutation diagrams. We provide explicit functional equations for the generating functions, with $k$ as a parameter.
\end{abstract}

\maketitle

\section{Introduction}\label{sec:intro}
An \emph{arc diagram} representation of a combinatorial class is
useful for highlighting the presence of certain kinds of patterns. It
is a graphical representation with labelled vertices, ordered along a
row. Different classes impose different restrictions on the allowable
edges and the vertex degrees. Some examples include matchings, set
partitions, permutations, and \emph{tangled diagrams}, which model RNA
secondary structures~\cite{ChQiRe08,BuMiPo10,Chetal07,ChHaRe09,deMi06,deMi07,JiQiRe08}. This representation can
facilitate random generation, enumeration, and the analysis of two key
pattern families, \emph{nestings\/} and \emph{crossings}.

Two arcs $(i_1, j_1)$ and $(i_2, j_2)$ \emph{nest} if
$i_1<i_2<j_2<j_1$. These two arcs \emph{cross} if
$i_1<i_2<j_1<j_2$. Many authors are interested in the total number of
crossings, and the case of exactly zero crossings seems to be
particularly compelling, across various structures, as we describe in
a moment.  A stronger notion is that of $k$-nesting pattern. A set of
$k$ arcs forms a \emph{$k$-nesting} if each of the $\binom{k}{2}$ pairs of
arcs nest. The definition of $k$-crossing is
analogous. Figure~\ref{fig:3nest} illustrates a $3$-nesting and a
$3$-crossing. This notion is slightly modified in the case of
permutations, as is described in Section~\ref{sec:permutations}. A diagram is
said to be \emph{$k$-nonnesting} if it does not contain a $k$-nesting
pattern. Again, we define a \emph{$k$-noncrossing} diagram analogously.

\begin{figure}[ht]
\begin{tikzpicture}[scale=0.6]
\node[pnt] at (0,0)(1){};
\node[pnt] at (1,0)(2){};
\node[pnt] at (2,0)(3){};
\node[pnt] at (3,0)(4){};
\node[pnt] at (4,0)(5){};
\node[pnt] at (5,0)(6){};
\node[pnt] at (10,0)(1b){};
\node[pnt] at (11,0)(2b){};
\node[pnt] at (12,0)(3b){};
\node[pnt] at (13,0)(4b){};
\node[pnt] at (14,0)(5b){};
\node[pnt] at (15,0)(6b){};
\draw(1)  to [bend left=45] (6);
\draw(2)  to [bend left=45] (5);
\draw(3)  to [bend left=45] (4);
\draw(1b)  to [bend left=45] (4b);
\draw(2b)  to [bend left=45] (5b);
\draw(3b)  to [bend left=45] (6b);
\end{tikzpicture}
\caption{A $3$-nesting and a $3$-crossing.}
\label{fig:3nest}
\end{figure}
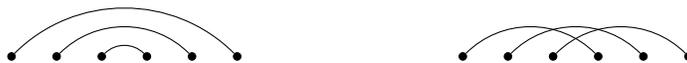

\subsection{Set partitions}
The arc diagram approach to the study of set partitions in
combinatorics was originally dominated by the study of noncrossing
diagrams~\cite{DeZa86,Si00,Kim11}, but the very novel, and
extremely robust bijection of Chen, Deng, Du, Stanley and
Yan~\cite{Chetal07} between matchings with maximum nesting size $k$
to matchings with maximum crossing size $k$ led to a very rich area of
study. The authors adapt the bijection to handle set partitions in the
same article, and later authors adapted the argument to graphs and
permutations. The heart of the argument maps the objects to a sequence
of Young tableaux, and the bijection is achieved by taking the transpose
of the tableaux.

From the enumerative point of view, results have been far less
forthcoming. The enumeration of \mbox{$k$-nonnesting} matchings was
completed by Chen \emph{et al.} in their master work by using a
bijection to collections of noncrossing Dyck paths whose
enumeration was known. Consequently, the results are elegant. Also,
the class of $2$-noncrossing (or simply noncrossing) partitions is
counted by Catalan numbers, and so one could hope for a similarly
beautiful enumeration scheme. So far, only exact formulas exist for
$3$-noncrossing set partitions, and beyond that, a far more complicated
structure is conjectured.  Specifically, Bousquet-M\'{e}lou and
Xin~\cite{BoXi05} did a complete functional equation analysis, and
determined explicit exact and asymptotic enumeration
formulas. Although their functional equations can be adapted to
enumerate set partitions with higher noncrossing numbers, they
conjecture that the structure of the generating function becomes more
complicated.
\begin{conjecture}[Bousquet-M\'elou,Xin 2005~\cite{BoXi05}]
\label{conj:nonDfin}
For every $k>3$, the generating function of $k$-noncrossing set partitions is not D-finite.
\end{conjecture}
Mishna and Yen~\cite{MiYe13} determined functional equations for
$k$-non\-nest\-ing set partitions, and described a process for isolating
coefficients, giving additional evidence for
Conjecture~\ref{conj:nonDfin}.

A more recent development  considers set partitions where the
maximum nesting size and the maximum crossing size are controlled
simultaneously. The resulting generating functions are rational~\cite{Marb13}, and
(in theory) can be determined explicitly. They can be summed
together for a different picture of generating functions
for~$k$-nonnesting partitions.

\subsection{Permutations}
The arc diagram is a convenient way to view permutations, as it
simultaneously highlights both the cycle structure, and the line
notation structure. Corteel~\cite{Cort07}
used arc diagrams to directly connect various permutation statistics,
such as exceedences and permutation patterns
to  occurrences of nestings and crossings. Burill, Mishna and
Post~\cite{BuMiPo10} extended this definition, and directly adapted
the results of Chen~\emph{et al.\/} to prove that $k$-nonnesting
permutations are in bijection with $k$-noncrossing permutations. They
computed enumerative data by brute force, by finding the nesting and crossing numbers
of each permutation. Consequently, the data is
restricted to small sizes only. The generating function for permutations
in which both the maximum nesting and maximum crossing numbers are controlled is rational~\cite{Yen13FPSAC},
similarly to the case of set partitions.

Extending the view of Corteel, we note that~$k$-nestings indicate
decreasing subsequences, and indeed some of our techniques here echo
studies of longest decreasing subsequences in permutations, especially
the recent work of Bousquet-M\'elou~\cite{Bous11}.

\subsection{Other classes}
Other classes have been enumerated, notably tangled
diagrams~\cite{CQR08} and RNA pseudoknots~\cite{MaRe08,SRSD11}, using
techniques related to those used for the enumeration of matchings.

\subsection{Generating trees}
Generating trees are a classic combinatorial tool,
introduced by West~\cite{West95} in the study of
pattern-avoiding permutations. Matchings, set partitions, and
permutations can all be described and generated using
generating trees, and the constructions translate directly to
generating function equations. In this strategy, each object in the
class is assigned a label (usually a natural combinatorial parameter),
and succession rules are given using a finite description process. The
class is then generated from some root, to which succession rules are
iteratively applied. The point of view we take here is best described
in~\cite{BaBoDeFlGaGo02}, although other examples of systems, such as
the ECO method~\cite{BrDePePi99} are related.  Our work most
closely resembles the recent generating tree approach of
Bousquet-M\'elou~\cite{Bous11} mentioned just above, as we use vectors
as labels, and the length of the vector comes directly from our
parameter of interest, the nesting number.

In order to translate a generating tree into generating function equations,
we demonstrate the following:
\begin{enumerate}
\item the generating tree construction generates each object uniquely;
\item each object $\pi$ is given a label $\ell(\pi)$;
\item the labels of the children of $\pi$ are completely determined by
  the label $\ell(\pi)$.
\end{enumerate}
Perhaps surprisingly, we are able to do this in all of the structures that we consider in this paper.

\subsection{Main results}
Our main contribution is a generating tree description of
$k$-nonnesting set partitions and $k$-nonnesting permutations, for
arbitrary $k$. Unlike previous enumerative analyses of these
objects, we do not make use of bijections with Young tableaux nor
lattice paths.  Our construction leads to a functional equation for
the generating function which we use to generate terms in the series.
We describe generic functional equations in a number of variables
which grows with~$k$. Thus far we have been unable to solve these
functional equations as our attempts to apply kernel methods have been obstructed
by the lack of symmetry in our description.

The key innovation in this study is a new class of structures that are
essentially arc diagrams ``under construction,'' which we call
\emph{open arc diagrams}. These are arc diagrams in which we allow
semi-arcs with a single endpoint. Usual arc diagrams then form the
subclass of diagrams with no semi-arcs.  We identify a minimal pattern to
avoid which prevents occurrences of $k$-nestings in the complete
diagrams.

Although our enumerative techniques are  subject to restrictions of
computer memory, they are sufficient to form the basis of some
conjectures: we have a precise conjecture on the asymptotic growth of
many $k$-nonnesting arc diagram families, including set partitions,
and permutations, and open diagram classes. This conjecture implies
that, asymptotically, the number of $k$-nonnesting open
diagrams divided by the number of $k$-nonnesting closed diagrams is
subexponential~\cite{Burr14}. This fact, which is not evident by looking only at small
sizes, implies that our open diagram construction is efficient
for large~$n$.

We also conjecture from our numerical data that there is a bijective correspondence between
open arc diagrams related to $3$-nonnesting set partitions and Baxter
permutations (Conjecture~\ref{conj:baxter}). This class does not
obviously posess the same symmetries as other classes counted by the
same sequence, and the generating tree is quite different from all
previously known examples. This is intriguing, although it complicates
the search for an easy combinatorial bijection.

\subsection{Organization of the paper}
We describe open arc diagrams for the case of set partitions and their generating trees in
Section~\ref{sec:openarc}. We also provide some basic results on
converting these generating trees to functional equations for the
generating function.

In Section~\ref{sec:setpartitions3} we give the first application of this construction
to determine a set of functional equations to enumerate~$3$-nonnesting
set partitions. The goal of this section is to provide a detailed
description of the method.

Section~\ref{sec:setpartitionsk} contains the construction for
$k$-nonnesting set partitions for general~$k$, and also the construction
for set partitions avoiding a related pattern, called enhanced
$k$-nestings.

In Section~\ref{sec:permutations}, using similar ideas and combining
the descriptions for usual and enhanced nestings in set partitions, we
construct a generating tree for $k$-nonnesting permutations and obtain
analogous functional equations.

Finally, in Section~\ref{sec:perspectives} we discuss applications of
the generated enumerative data, and how this construction can be
modified to consider other arc diagram families, in particular RNA
secondary structures.

\section{Open arc diagrams: An object under construction}
\label{sec:openarc}

In this section we introduce a generalization of arc diagrams in which not all arcs have a right endpoint. These diagrams are then used as our basic example to illustrate how generating trees are constructed to  generate a combinatorial class
and to obtain an equation for its generating function. We conclude the section by defining nestings in these generalized arc diagrams.

\subsection{Arc diagrams}
An arc diagram of size $n$ is an embedded graph with vertices 1 to~$n$
drawn in an increasing row. We take the convention of labelling our
diagrams from left to right, so we can refer to left and right
endpoints of an arc.  We first focus on the case of arc diagrams
representing set partitions, which we call \emph{partition diagrams}
for short. In this case, the arcs are always drawn above the vertices,
and the partition block~$\{a_1, a_2, \dots, a_j\}$, where
$a_1<a_2<\dots<a_j$, is represented by the arcs~$(a_1, a_2), (a_2,
a_3), \dots, (a_{j-1}, a_j)$.

A vertex in a partition diagram is one of four types:
\smallskip

\begin{center}{\begin{tabular}{lclcl}
{\bf 1. fixed point} &\begin{tikzpicture} \node[pnt] at
    (0,0){}; \end{tikzpicture}& no incident edges;\\
{\bf 2. opener}& \begin{tikzpicture} \node[pnt] at (0,0){}; \draw[bend
    left=45](0,0) to (0.25, 0.25); \end{tikzpicture} & outdegree one,
  the left endpoint of an arc; \\
{\bf 3. transitory} &\begin{tikzpicture}\node[pnt] at (0,0){};\draw[bend
    left=45](-0.25, 0.25) to (0,0); \draw[bend left=45](0,0) to (0.25,
    0.25);\end{tikzpicture}& indegree and outdegree one, \\ &&the right endpoint of one arc and the left
endpoint of another;\\
{\bf 4. closer}&\begin{tikzpicture}\node[pnt] at (0,0){}; \draw[bend
    left=45](-0.25, 0.25) to (0,0);\end{tikzpicture}& indegree one, the right endpoint of an arc.
\end{tabular}}
\end{center}
\smallskip
Figure~\ref{fig:partition} shows the
diagram of a partition of $\{1, \dots, 9\}$, and illustrates the
different types of vertices.

\begin{figure}
  \begin{tikzpicture}[scale=0.5]
   \foreach \i in {1,...,9}
        \node[pnt,label=below:$\i$] at (\i,0)(\i) {};
   \draw(1)  to [bend left=45] (3);
   \draw(3)  to [bend left=45] (5);
   \draw(4)  to [bend left=45] (6);
   \draw(7)  to [bend left=45] (8);
   \draw(8)  to [bend left=45] (9);
   \end{tikzpicture}
\caption{The arc diagram representation of the partition $\{1,3,5 \}\{2\}\{4,6\}\{7,8,9\}$. The vertex
labelled $2$ is a \emph{fixed point}; the vertices labelled $1$, $4$, and $7$
are \emph{openers}; vertices $3$ and $8$
are \emph{transitories}; and $5$, $6$, and $9$ are \emph{closers}. }
\label{fig:partition}
\end{figure}
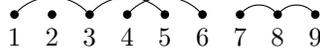
%

\subsection{Open arc diagrams}
To generalize partition diagrams, we allow semi-arcs with a left endpoint,
but no right endpoint. Otherwise said, we allow opener and
transitory vertices with incident arcs that are not `closed', which we call
open semi-arcs or simply semi-arcs. We call these generalized diagrams
\emph{open partition diagrams}.  To ensure a standard representation,
we draw the semi-arcs to a vertical line to the right of vertex $n$,
and retain their order, not allowing the semi-arcs to intersect. We
denote the semi-arc with left endpoint $i$ by
$(i,*)$. Figure~\ref{fig:halfarc} contains an example.

A diagram with no semi-arcs represents a usual set partition. We call
such a diagram a \emph{complete} (partition) diagram.  We view an open
partition diagram as a future set partition, or a set partition in
progress, specifically in a process which incrementally adds vertices
in numerical order, possibly closing semi-arcs and/or opening new
ones. The open partition diagram $\pi$ in Figure~\ref{fig:halfarc} is
an ancestor to the two set partitions represented in
Figure~\ref{fig:descendantshalfarc}, in this process among infinitely
many others.
\begin{figure}
  \begin{tikzpicture}[scale=0.3]
   \node at (-1,1){$\pi=$};
   \foreach \i in {1,...,9}
        \node[pnt,label=below:$\i$] at (\i,0)(\i) {};
   \node[opnt] at (10,1)(1a){};
   \node[opnt] at (10,1.5)(2a){};
   \node[opnt] at (10,2)(3a){};
   \draw(1)  to [bend left=45] (3);
   \draw(3)  to [bend left=40] (3a);
   \draw(4)  to [bend left=45] (6);
   \draw(5)  to [bend left=40] (2a);
   \draw(7)  to [bend left=40] (1a);
   \draw(8)  to [bend left=45] (9);
\end{tikzpicture}

\begin{tikzpicture}[scale=0.4]
   \foreach \i in {1,...,12}
        \node[pnt,label=below:$\i$] at (\i,0)(\i) {};
   \draw(1)  to [bend left=45] (3);
   \draw(3)  to [bend left=45] (12);
   \draw(4)  to [bend left=45] (6);
   \draw(5)  to [bend left=45] (11);
   \draw(7)  to [bend left=45] (10);
   \draw(8)  to [bend left=45] (9);
\end{tikzpicture}
\hspace{2cm}
\begin{tikzpicture}[scale=0.4]
   \foreach \i in {1,...,13}
        \node[pnt,label=below:$\i$] at (\i,0)(\i) {};
   \draw(1)  to [bend left=45] (3);
   \draw(3)  to [bend left=45] (10);
   \draw(4)  to [bend left=45] (6);
   \draw(5)  to [bend left=45] (11);
   \draw(7)  to [bend left=45] (12);
   \draw(8)  to [bend left=45] (9);
   \draw(10) to [bend left=45](13);
\end{tikzpicture}
\caption{An example of an open partition diagram, $\pi$, and two set partitions with $\pi$ as an ancestor.}
\label{fig:descendantshalfarc}
\label{fig:halfarc}
\end{figure}
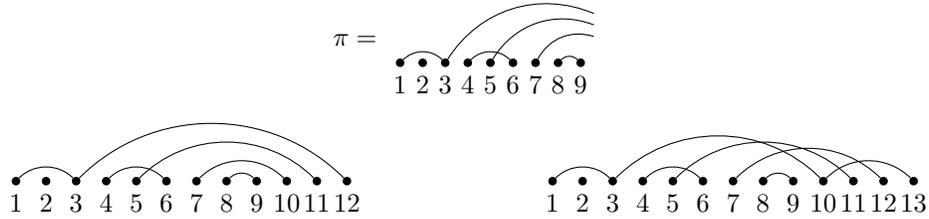

An open partition diagram could also be viewed as a set partition in
which each block is coloured one of two colours: one for proper
blocks, i.e. fixed points or those that end in a closer, and another for blocks ending with a semi-arc. For example,
$\pi$ above represents the bicoloured partition $\mathbf{\{1,3\}},
\{2\}, \{4,6\}, \mathbf{\{5\}, \{7\}}, \{8,9\}$, where blocks written
in bold face are those that end in a semi-arc, and normal fonts
indicate proper blocks.

\subsection{Generating tree for open partition diagrams}
\label{sec:openarcgentree}
The generating tree construction in this case is straightforward.
Given an open partition diagram with $n$ vertices, the added vertex $n+1$ can be  any
of the four kinds, with the caveat that it can be a closer or transitory only if there
is an existing semi-arc to be closed.

For example, the open partition diagram in Figure~\ref{fig:halfarc} generates the 8
diagrams in Figure~\ref{fig:childrenhalfarc}, which we call its children.
\begin{figure}\centering
\begin{tikzpicture}[scale=0.25]
 \diagpi
\end{tikzpicture}
\centerline{\Huge$\downarrow$}
\begin{tikzpicture}[scale=0.25]
   \foreach \i in {1,...,10}
        \node[pnt,label=below:$\i$] at (\i,0)(\i) {};
   \node[opnt] at (11,1)(1a){};
   \node[opnt] at (11,1.5)(2a){};
   \node[opnt] at (11,2)(3a){};
   \node[opnt] at (11,2.5)(4a){};
   \draw(1)  to [bend left=45] (3);
   \draw(3)  to [bend left=45] (3a);
   \draw(4)  to [bend left=45] (6);
   \draw(5)  to [bend left=45] (2a);
   \draw(7)  to [bend left=45] (1a);
   \draw(8)  to [bend left=45] (9);
\end{tikzpicture}\\
\begin{tikzpicture}[scale=0.25]
   \foreach \i in {1,...,10}
        \node[pnt,label=below:$\i$] at (\i,0)(\i) {};
   \node[opnt] at (11,1)(1a){};
   \node[opnt] at (11,1.5)(2a){};
   \node[opnt] at (11,2)(3a){};
   \node[opnt] at (11,2.5)(4a){};
   \draw(1)  to [bend left=45] (3);
   \draw(3)  to [bend left=45] (4a);
   \draw(4)  to [bend left=45] (6);
   \draw(5)  to [bend left=45] (3a);
   \draw(7)  to [bend left=45] (2a);
   \draw(8)  to [bend left=45] (9);
   \draw(10) to [bend left=45] (1a);
\end{tikzpicture}\\
\begin{tikzpicture}[scale=0.25]
   \foreach \i in {1,...,10}
        \node[pnt,label=below:$\i$] at (\i,0)(\i) {};
   \node[opnt] at (11,1)(1a){};
   \node[opnt] at (11,1.5)(2a){};
   \node[opnt] at (11,2)(3a){};
   \draw(1)  to [bend left=45] (3);
   \draw(3)  to [bend left=45] (10);
   \draw(4)  to [bend left=45] (6);
   \draw(5)  to [bend left=45] (3a);
   \draw(7)  to [bend left=45] (2a);
   \draw(8)  to [bend left=45] (9);
   \draw(10)  to [bend left=45] (1a);
\end{tikzpicture}
\begin{tikzpicture}[scale=0.25]
   \foreach \i in {1,...,10}
        \node[pnt,label=below:$\i$] at (\i,0)(\i) {};
   \node[opnt] at (11,1)(1a){};
   \node[opnt] at (11,1.5)(2a){};
   \node[opnt] at (11,2)(3a){};
   \draw(1)  to [bend left=45] (3);
   \draw(3)  to [bend left=45] (3a);
   \draw(4)  to [bend left=45] (6);
   \draw(5)  to [bend left=45] (10);
   \draw(7)  to [bend left=45] (2a);
   \draw(8)  to [bend left=45] (9);
   \draw(10)  to [bend left=45] (1a);
\end{tikzpicture}
\begin{tikzpicture}[scale=0.25]
   \foreach \i in {1,...,10}
        \node[pnt,label=below:$\i$] at (\i,0)(\i) {};
   \node[opnt] at (11,1)(1a){};
   \node[opnt] at (11,1.5)(2a){};
   \node[opnt] at (11,2)(3a){};
   \draw(1)  to [bend left=45] (3);
   \draw(3)  to [bend left=45] (3a);
   \draw(4)  to [bend left=45] (6);
   \draw(5)  to [bend left=45] (2a);
   \draw(7)  to [bend left=45] (10);
   \draw(8)  to [bend left=45] (9);
   \draw(10)  to [bend left=45] (1a);
\end{tikzpicture}\\
\begin{tikzpicture}[scale=0.3]
   \foreach \i in {1,...,10}
        \node[pnt,label=below:$\i$] at (\i,0)(\i) {};
   \node[opnt] at (11,1)(1a){};
   \node[opnt] at (11,1.5)(2a){};
   \draw(1)  to [bend left=45] (3);
   \draw(3)  to [bend left=45] (10);
   \draw(4)  to [bend left=45] (6);
   \draw(5)  to [bend left=45] (2a);
   \draw(7)  to [bend left=45] (1a);
   \draw(8)  to [bend left=45] (9);
\end{tikzpicture}
\begin{tikzpicture}[scale=0.3]
   \foreach \i in {1,...,10}
        \node[pnt,label=below:$\i$] at (\i,0)(\i) {};
   \node[opnt] at (11,1)(1a){};
   \node[opnt] at (11,1.5)(2a){};
   \draw(1)  to [bend left=45] (3);
   \draw(3)  to [bend left=45] (2a);
   \draw(4)  to [bend left=45] (6);
   \draw(5)  to [bend left=45] (10);
   \draw(7)  to [bend left=45] (1a);
   \draw(8)  to [bend left=45] (9);
\end{tikzpicture}
\begin{tikzpicture}[scale=0.3]
   \foreach \i in {1,...,10}
        \node[pnt,label=below:$\i$] at (\i,0)(\i) {};
   \node[opnt] at (11,1)(1a){};
   \node[opnt] at (11,1.5)(2a){};
   \draw(1)  to [bend left=45] (3);
   \draw(3)  to [bend left=45] (2a);
   \draw(4)  to [bend left=45] (6);
   \draw(5)  to [bend left=45] (1a);
   \draw(7)  to [bend left=45] (10);
   \draw(8)  to [bend left=45] (9);
\end{tikzpicture}\\
\caption{The open partition diagram for $\pi$ and its eight children.}
\label{fig:childrenhalfarc}
\end{figure}
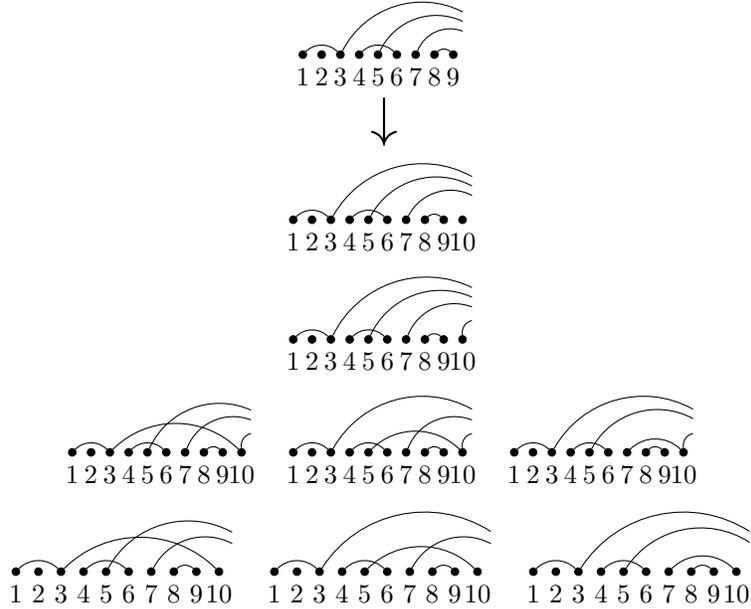
The label of such a diagram is the number of open semi-arcs. This
parameter is sufficient to describe the number of children and all
their labels.  Indeed, suppose a diagram~$\pi$ with~$n$ vertices
has~$\ell(\pi)=m$ semi-arcs (coming from either opener or transitory
vertices). Its number of children and their labels (i.e., number of
semi-arcs) are as follows, depending on the type of the added
vertex~$n+1$:
\begin{description}
\item[1. fixed point] one child with~$m$ semi-arcs;
\item[2. opener] one child with~$m+1$ semi-arcs;
\item[3. transitory] $m$ children, each with $m$ semi-arcs;
\item[4. closer] $m$ children, each with~$m-1$ semi-arcs.
\end{description}
Remark that the last two are trivially empty when~$m=0$.  This sums to
a total of~$2m+2$ children for any diagram with~$m$ semi-arcs. The
fact that the number of children of~$\pi$ and their labels are
completely determined by~$\ell(\pi)$ makes these structures suitable
for generating tree techniques.  The good news is that this property
remains true when we integrate the~$k$-nonnesting constraint.  To
capture the more complex cases, we consider labels which are vectors.

For the above tree, labels have only one component, and the succession
rule is described in the following format:
\begin{equation}\label{eq:easytree}
  [m]\rightarrow
  \begin{array}{llcr}
     ~[m],   & &\quad&\text{(fixed point)}\\[2mm]
     ~[m+1], & &&\text{(opener)}\\[2mm]
     ~\underbrace{[m],[m], \dots,  [m]}_{m\text{ copies}},&\text{if $m>0$,}&&\text{(transitory)}\\[2mm]
     ~\underbrace{[m-1],  [m-1],\dots,   [m-1]}_{m \text{ copies}}, &\text{if $m>0$.}&&\text{(closer)}\\[2mm]
  \end{array}
\end{equation}
We denote the set of labels in the succession rule by
$\Succ([m])$. The root of the tree has label~$[0]$, as the empty set
partition has no semi-arcs.

For comparison, we point out that the rewriting system of Banderier
\emph{et al.} in~\cite{BaBoDeFlGaGo02} uses the parameter ``number of
children'' as the label, which equals~$2m+2$ for a diagram with~$m$
semi-arcs.

\subsection{Converting generating tree specifications to generating
  function equations}
\label{sec:BaseCaseConversion}
Often, the process of translating a generating tree specification to a
generating function equation is quite straightforward. Let
$\mathcal{P}$ be a combinatorial class, and let~$\ell(\pi)=[r_1,
\dots, r_k]$ be the vector label of $\pi\in\mathcal{P}$. If the label
of an object is a vector of length~$k$, we construct a functional
equation using ~$k+1$ variables. The variable~$z$ marks the depth (or
level) of the object in the generating tree, which is given by its
size.  Define the multi-variate generating function
\[
P(\mathbf{u}; z)= \sum_{\pi\in \mathcal{P}} \mathbf{u}^{\ell(\pi)}z^{|\pi|}=\sum_{\pi\in\mathcal{P}} u_1^{r_1} u_2^{r_2}\dots u_k^{r_k}\, z^{|\pi|},
\]
where ${\mathbf u}=(u_1, u_2, \dots, u_k)$, and $|\pi|$ denotes the size
of~$\pi$. In a proper generating tree construction, an
element either is the root~$\pi_0$, or it has a unique parent,
which generates it exactly once. Thus, if $\mathcal{C}(\pi)$ denotes
the set of children of $\pi$, then we deduce the following
relationship:
\[
P(\mathbf{u}; z)= \mathbf{u}^{\ell(\pi_0)} z^{|\pi_0|} + \sum_{\pi\in
  \mathcal{P}} z^{|\pi|+1}\sum_{\pi'\in\mathcal{C}(\pi)}
\mathbf{u}^{\ell(\pi')}.
\]
Recall that $\ell(\pi')$ is determined from $\ell(\pi)$ by definition.
When this correspondence is algebraically straightforward, we can
describe a functional equation.

As we see in the next section, we can treat exponential
generating functions similarly:
\[
\sum_{\pi\in \mathcal{P}} \mathbf{u}^{\ell(\pi)}\,
\frac{z^{|\pi|}}{|\pi|!} = \mathbf{u}^{\ell(\pi_0)} z^{|\pi_0|} + \sum_{\pi\in
  \mathcal{P}} \frac{z^{|\pi|+1}}{(|\pi|+1)!}\sum_{\pi'\in\mathcal{C}(\pi)} \mathbf{u}^{\ell(\pi')}.
\]

\subsection{Exponential generating function for open partition diagrams}
In an open partition diagram, one can distinguish between blocks which
end with a semi-arc and those which do not. The decomposable structure
description of set partitions, which in the methodology
of~\cite{FlSe09} is presented as the labelled class
$\operatorname{Set}(\mathcal{B})$, where
$\mathcal{B}=\operatorname{Set_{card \neq 0}}(\mathcal{Z})$, is easily
modified to accommodate the two types of blocks, giving the class
$\mathcal{P}=\operatorname{Set}(\mathcal{B}+\mathcal{B})$.  From this
combinatorial equation, it is clear that the exponential generating
function for open partition diagrams is $e^{2(e^z-1)}$.

Alternatively, this formula can be obtained by translating the
succession rule~\eqref{eq:easytree} into a differential equation for
the exponential generating function and solving it. To extend to
the~$k$-nonnesting case we persue the strategy of translating the
succession rule into generating function equations. For these cases we
use the equations to generate coefficients rather than solving them.
For this reason, it is more convenient for our purposes to use
ordinary generating functions instead of exponential ones.

Let us now find the bivariate generating function~$P(u,z)$ where the
exponent of $u$ is the label of the node, which equals the number of
open semi-arcs.  Let $\pi$ be an open partition diagram with $|\pi|=n$
vertices and $\ell(\pi)=m$ open semi-arcs.  Recall that its children
$\pi'\in\mathcal{C}(\pi)$ in the generating tree can be of four types,
depending on the type of the added vertex $n+1$, and that the labels
of these children are described by~\eqref{eq:easytree}. It follows
that

\[\sum_{\pi'\in\mathcal{C}(\pi)} u^{\ell(\pi')} =
\underbrace{u^m}_{\text{fixed point}} +
\underbrace{u^{m+1}}_{\text{opener}} +
\underbrace{m\,u^m}_{\text{transitory}} +
\underbrace{m\,u^{m-1}}_{\text{closer}},
\]
which gives the following generating function recurrence:
\begin{eqnarray*}
P(u,z)&=&\sum_{n,m}p(m,n) \frac{u^mz^n}{n!} \\
&=& 1+  \sum_{n,m} p(m,n) \frac{z^{n+1}}{(n+1)!}\left( u^m+u^{m+1}+m u^m
  +m u^{m-1} \right)\\
&=& 1+ \int \left(P(u,z)+ uP(u,z)+ uP_u(u,z)+P_u(u,z)\right)\, dz,
\end{eqnarray*}
where we use the fact that $\int P(z,u)\, dz =\sum p(k,n) \frac{1}{n+1} \frac{z^{n+1}u^k}{n!}$.
Differentiating with respect to $z$ we get
$$P_z(u,z)= (1+u)P_u(u,z) + (1+u)P(u,z) = (1+u)\left(P_u(u,z) + P(u,z)\right),$$
and solving this differential equation gives
$P(u,z)=e^{(1+u)(e^z-1)}$.

\subsection{Nestings and future nestings}
\label{sec:nesting}
Recall that a $k$-nesting in a partition diagram is a set of $k$ mutually nesting arcs, that is, arcs $(i_1, j_1),\dots,(i_{k}, j_{k})$ such that
\[
   i_1<i_2<\dots<i_k<j_k<j_{k-1}<\dots<j_1.
\]
To incorporate the nesting constraint to the generating tree, we need to generalize the notion of $k$-nestings to open partition diagrams.

As before, a \emph{(regular) $k$-nesting} in an open partition diagram is a set of
$k$ mutually nesting (closed) arcs. To handle semi-arcs, we introduce the concept
of a \emph{future $k$-nesting}.

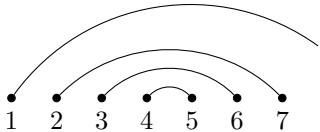
\begin{figure}[htb]
\begin{tikzpicture}[scale=0.6]
   \foreach \i in {1,...,7}
        \node[pnt,label=below:$\i$] at (\i,0)(\i) {};
   \node[opnt] at (8,1)(1a){};
   \draw(1)  to [bend left=45] (1a);
   \draw(2)  to [bend left=45] (7);
   \draw(3)  to [bend left=45] (6);
   \draw(4)  to [bend left=45] (5);
\end{tikzpicture}
\caption{An example of a future $4$-nesting.}
\label{fig:futurenesting}
\end{figure}

\begin{defn}A \emph{\bf future $k$-nesting} is a set of~$k-1$ mutually
  nesting arcs and one semi-arc, such that the left end-point of the
  semi-arc is to the left of the $k-1$-nesting. \end{defn}

An example is drawn in Figure~\ref{fig:futurenesting}. Recall that
since semi-arcs do not intersect, a semi-arc that is above another one
also has its left endpoint further to the left. The following claim
follows trivially.

\begin{claim}\label{claim}
If a semi-arc belongs to a future $k$-nesting, then any semi-arc above it also
belongs to a future $k$-nesting.
\end{claim}

Note that having multiple semi-arcs above a $k-1$-nesting does not imply
$\ell$-nestings in its descendants for $\ell>k$.  For example, the
three semi-arcs in Figure~\ref{fig:halfarc}, together with the arc
under all three of them, become a $4$-nesting in the partition on the left
of Figure~\ref{fig:descendantshalfarc}, but they become three separate
$2$-nestings in the partition on the right.

\begin{defn} An open partition diagram is \emph{$k$-nonnesting} if it
contains neither regular nor future $k$-nestings.
\end{defn}

Note that, in particular, $k$-nonnesting complete partition diagrams are simply $k$-nonnesting partitions, since the
condition about future $k$-nestings is void when there are no semi-arcs.

Our strategy is to pare the general open partition diagram by removing
the vertices which introduce future $k$-nestings, and all of their
children. Next we show that this is sufficient. The main challenge treated by the
remaining sections is to prove that a label exists from which we can
build  the tree.

\begin{prop} \label{prop:subclass} Consider the subclass of open
  partition diagrams generated from the empty diagram (as described in
  Section~\ref{sec:openarcgentree}) left after
  pruning all subtrees of diagrams with a future~$k$-nesting. The
  elements in this class are precisely the $k$-nonnesting open partition diagram.
  In particular, the complete diagrams in this class are precisely the $k$-nonnesting
  partitions.
\end{prop}
\begin{proof}
Any open partition diagram that is a descendant of a diagram with a future $k$-nesting has a
either a $k$-nesting or a future $k$-nesting. Indeed, the only way to remove a future $k$-nesting is by closing its top semi-arc,
which creates a $k$-nesting.

On the other hand, starting from a diagram with a $k$-nesting, and deleting vertices starting from
the right, one can find at least one ancestor diagram that contains a
future $k$-nesting. Thus, in order to obtain all $k$-nonnesting open partition diagrams (and thus all $k$-nonnesting complete partition diagrams),
it is sufficient to generate all open diagrams that avoid future $k$-nestings.
\end{proof}

By generating all $k$-nonnesting open partition diagrams, our process generates a superset of the $k$-nonnesting
partitions. In the generating function equations, we recover the
generating function for our desired class by a variable
specialization.  Mishna and Yen~\cite{MiYe13} gave a construction
which generates only this class, but our construction has two
important advantages over theirs: it can handle the enhanced case (see
Section~\ref{sec:enhanced}), and it can be extended to $k$-nonnesting
permutations (see Section~\ref{sec:permutations}).

However, we can quantify how much larger the set of $k$-nonnesting
open diagrams is compared to the set of closed diagrams--- in other
words, how much rejection would be necessary in a generation
scheme. Numerically, we observed that, perhaps surprisingly, this
ratio is not substantial for~$n$ large.  This led to a combinatorial
analysis and an upper bound on the exponential growth rate for the
number of $k$-nonnesting diagrams~\cite{Burr14}.  Based on this
construction, which applies also to more general classes, Burrill
makes the following conjecture.
\begin{conjecture}[\cite{Burr14}] Let $p_k(n)$ and $o_k(n)$ denote the
  number of $k$-nonnesting set partitions and open partition diagrams
  of size $n$, respectively. Then
\[
\lim_{n \rightarrow \infty} p_k(n)^{\frac{1}{n}}=\lim_{n\rightarrow
  \infty} o_k(n)^{\frac{1}{n}}= k^2.
\]
\end{conjecture}

\section{$3$-nonnesting set partitions}
\label{sec:setpartitions3}
We start with the $3$-nonnesting case as it forms the general template
for both $k$-nonnesting partitions, and also $k$-nonnesting
permutations. It is similar in structure to the open partition case.

This section is divided as follows: we begin with a description of the
generating tree construction for $3$-nonnesting open partition
diagrams.  The complete diagrams in this class are the $3$-nonnesting
set partitions. Then, we prove that the labels of the children of an
object can be determined from that object's label. We conclude with a
translation of the generating tree into generating function equations,
and compute some enumerative data.

\subsection{Generating tree construction}
\label{sec:gentree3}
By Proposition~\ref{prop:subclass}, the generating tree of
$3$-nonnesting open partition diagrams is a subtree of the tree in
Section~\ref{sec:openarcgentree}, which generates all open partition
diagrams. We can describe an appropriate label for this subtree.

To each $3$-nonnesting open partition diagram $\pi$ we associate the
label $\ell(\pi)=[m,s]$ where $m$ is the total number of semi-arcs and
$s$ is the number of semi-arcs in a future $2$-nesting.  The latter
value is the number of semi-arcs above at least one closed arc.
Recall from Claim~\ref{claim} if a semi-arc belongs to a future
$2$-nesting, then any semi-arc above it also belongs to a future
$2$-nesting. The label of the empty partition diagram $\epsilon$ is
$\ell(\epsilon)=[0,0]$.

\begin{example} In Figure~\ref{fig:labeldiagram}, the two arrows
  indicate the semi-arcs in future $2$-nestings. Thus, the label of
  this diagram is $[4,2]$.  If vertex~12 is a closer or a transitory
  vertex closing the semi-arc started at vertex 7, then $(7,12)$,
  $(8,9)$ and $(3,*)$ form a future $3$-nesting, which is forbidden in
  the construction. On the other hand, if vertex $12$ closes the
  semi-arc started at vertex 3, then only a $2$-nesting is created,
  but no future $3$-nesting.  Vertex $12$ can also close any of the
  other two semi-arcs and remain in the class. Consequently, this arc
  diagram has eight children in our generating tree: the two obtained
  from adding a fixed point or opener, plus the six diagrams obtained
  by making vertex $12$ a closer or transitory vertex.
\begin{figure}[htb]
\begin{tikzpicture}[scale=0.4]
  \foreach \i in {1,...,11}
       \node[pnt,label=below:$\i$] at (\i,0)(\i) {};
  \node[opnt] at (12,1)(1a){};
  \node[opnt] at (12,1.5)(2a){};
  \node[opnt,label=right:$\leftarrow$] at (12,2)(3a){};
  \node[opnt,label=right:$\leftarrow$] at (12,2.5)(4a){};
  \draw(1)  to [bend left=45] (3);
  \draw(3)  to [bend left=45] (4a);
  \draw(4)  to [bend left=45] (6);
  \draw(5)  to [bend left=45] (8);
  \draw(7)  to [bend left=45] (3a);
  \draw(8)  to [bend left=45] (9);
  \draw(10)  to [bend left=45] (2a);
  \draw(11)  to [bend left=45] (1a);
  \end{tikzpicture}
\caption{An open partition diagram with label $[4,2]$.}
\label{fig:labeldiagram}
\end{figure}
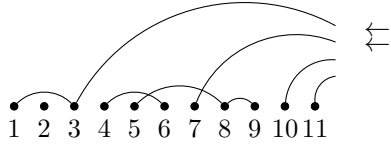
\end{example}

In general, to avoid creating a future $3$-nesting, we must not close
semi-arcs in future $2$-nestings, with the exception of the top arc,
which can always be closed. Semi-arcs not belonging to future
$2$-nestings can be closed as well.  Finally, adding openers and fixed
points does not create future $3$-nestings, so there are no
restrictions on these types of vertices.

\subsection{Succession rule}
Next we show that the label of a node provides sufficient information to
determine the number of children and their labels.

Suppose a diagram $\pi$ of size $n$ has label
$\ell(\pi)=[m,s]$. Its children (and some of their labels) are as follows, depending on the type of the added vertex $n+1$:
\begin{description}
\item[1. fixed point] one child with label $[m,s]$;
\item[2. opener] one child with label $[m+1,s]$;
\item[3. transitory] $m-s+1$ (if $s>0$) or $m$ (if $s=0$) children,
  since we can close any of the $m-s$ semi-arcs not in future
  $2$-nestings, plus the top semi-arc in the case $s>0$;
\item[4. closer] $m-s+1$ (if $s>0$) or $m$ (if $s=0$) children.
\end{description}
In the next theorem, we address the labels of children. To build
intuition, we first return to the example.
\begin{example}
We continue the example in Figure~\ref{fig:labeldiagram}. The label of
the diagram is $[4,2]$ and the application of the succession rule
yields the following, sorted by the type of vertex added:
\begin{description}
\item[1. fixed point] label $[4,2]$;
\item[2. opener] label $[5,2]$;
\item[3. transitory] in each one of the three children, the number of
  semi-arcs is preserved, while the number of semi-arcs belonging to
  future $2$-nestings depends on the semi-arc that is closed, giving
  labels $[4,3]$, $[4,2]$ and $[4,1]$ when closing $(11,*)$, $(10,*)$
  and $(3,*)$, respectively;
\item[4. closer] in each one of the three children, the number of
  semi-arcs is reduced by one, but otherwise it is analogous to the
  transitory case, giving labels $[3,3]$, $[3,2]$, $[3,1]$ when
  closing $(11, *)$, $(10, *)$ and $(3, *)$, respectively.
\end{description}
\end{example}
\begin{theorem}\label{thm:3part}
Let $\Pi^{(2)}$ be the set of $3$-nonnesting open partition diagrams. To each diagram, associate the label
$\ell(\pi)=[m,s]$ if $\pi$ has $m$ semi-arcs, $s$ of which belong to some
future $2$-nesting. Then the number of diagrams in $\Pi^{(2)}$ of size
$n$ is the number of nodes at level $n$ in the generating tree with
root label $[0,0]$ for $n=0$, and succession rule given by
\begin{equation*}
  [m,s]\rightarrow
  \begin{array}{llr}
     ~[m,s],   & &\text{(fixed point)}\\[2mm]
     ~[m+1,s], & &\text{(opener)}\\[2mm]
     ~[m, s], [m,s+1], \dots,  [m, m-1],&\text{if $m>0$,}&\text{(transitory)}\\[2mm]
     ~[m-1, s],  [m-1,s+1],\dots,   [m-1, m-1], &\text{if $m>0$,}&\text{(closer)}\\[2mm]
      ~[m, s-1], [m-1,s-1].&\text{if $m>0$ and $s>0$.}& \text{(transitory and closer}\\
           &&\text{respectively, when $s>0$)}
  \end{array}
\end{equation*}
The number of~$3$-nonnesting set partitions of~$\{1, \dots, n\}$ is
equal to the number of nodes with label~$[0,0]$ at level~$n$.
\end{theorem}

\begin{proof}
We focus on closer vertices, since transitory vertices behave identically
except for the first component, and fixed points and openers have already been discussed.

Consider a $3$-nonnesting open partition diagram $\pi$ with label $[m,s]$. It must be that $m>0$, or else
there is nothing to close.
 Closing any semi-arc  decreases  the  total  number  of  semi-arcs  by
one. Closing the bottom semi-arc  turns all the semi-arcs above it
into future  $2$-nestings, producing a diagram with label  $[m-1, m-1]$.
Closing the second lowest semi-arc  converts all remaining semi-arcs except
the bottom one into future $2$-nestings, and the resulting diagram
has label $[m-1,  m-2]$. More generally, each of the $m-s$ semi-arcs not belonging to future
$2$-nestings can be closed in this manner, yielding $m-s$ diagrams with labels $[m-1,m-i]$ for $1\le i\le m-s$.

Finally, if $s>0$, then $\pi$ contains a future $2$-nesting, and the top
semi-arc can  be closed without  creating a future  $3$-nesting. This operation removes
a future  $2$-nesting has, resulting in label $[m-1, s-1]$.
\end{proof}

The first few levels of this generating tree scheme are shown in Figure~\ref{fig:parttree}.

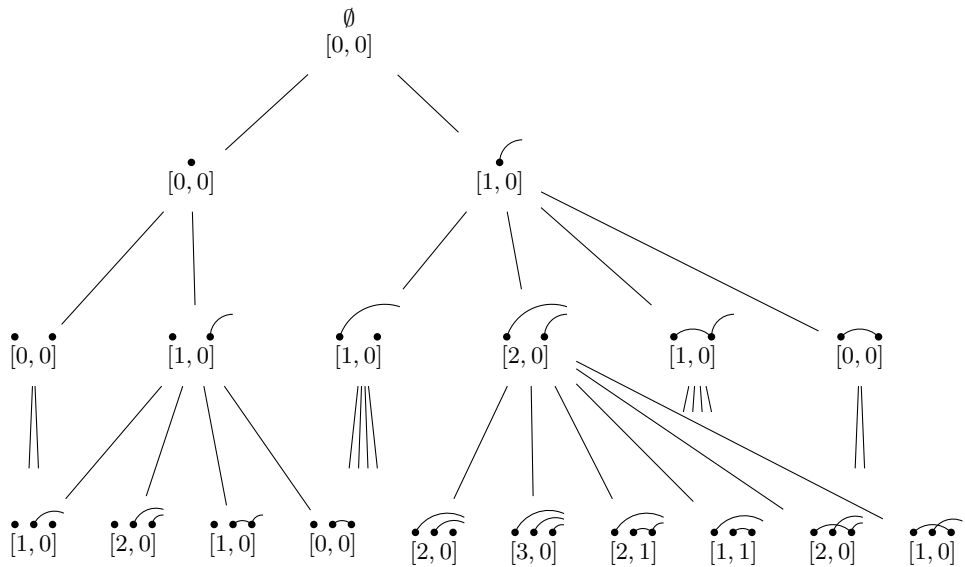
\begin{figure}[htb]
\small\centering
\newsavebox{\Paone}%
\sbox{\Paone}{
   \begin{tikzpicture}[part]
      \node {$\emptyset$};
      \node[partlabel,below=3pt] {$[0,0]$};
   \end{tikzpicture}
}
\newsavebox{\Pbone}%
\sbox{\Pbone}{
   \begin{tikzpicture}[part]
      \node[partition] {};
      \node[partlabel] {$[0,0]$};
   \end{tikzpicture}%
}
\newsavebox{\Pbtwo}%
\sbox{\Pbtwo}{
   \begin{tikzpicture}[part]%
      \node[partition] at (0,0) {};
      \draw (0,0) to (0.6,0.6);
      \node[partlabel] {$[1,0]$};
   \end{tikzpicture}%
}
\newsavebox{\Pcone}%
\sbox{\Pcone}{
   \begin{tikzpicture}[part]%
      \foreach \i in {0,1} \node[partition] at (\i,0) {};
      \node[partlabel] at (0.5,0) {$[0,0]$};
   \end{tikzpicture}%
}
\newsavebox{\Pctwo}%
\sbox{\Pctwo}{
   \begin{tikzpicture}[part]%
      \foreach \i in {0,1} \node[partition] at (\i,0) {};
      \node[partlabel] at (0.5,0) {$[1,0]$};
      \draw (1,0) to (1.6,0.6);
   \end{tikzpicture}%
}
\newsavebox{\Pcthree}%
\sbox{\Pcthree}{
   \begin{tikzpicture}[part]%
      \foreach \i in {0,1} \node[partition] at (\i,0) {};
      \node[partlabel] at (0.5,0) {$[1,0]$};
      \draw (0,0) to (1.6,0.8);
   \end{tikzpicture}%
}
\newsavebox{\Pcfour}%
\sbox{\Pcfour}{
   \begin{tikzpicture}[part]%
      \foreach \i in {0,1} \node[partition] at (\i,0) {};
      \node[partlabel] at (0.5,0) {$[2,0]$};
      \draw (0,0) to (1.6,0.8);
      \draw (1,0) to (1.6,0.6);
   \end{tikzpicture}%
}
\newsavebox{\Pcfive}%
\sbox{\Pcfive}{
   \begin{tikzpicture}[part]%
      \foreach \i in {0,1} \node[partition] at (\i,0) {};
      \node[partlabel] at (0.5,0) {$[1,0]$};
      \draw (0,0) to (1,0);
      \draw (1,0) to (1.6,0.6);
   \end{tikzpicture}%
}
\newsavebox{\Pcsix}%
\sbox{\Pcsix}{
   \begin{tikzpicture}[part]%
      \foreach \i in {0,1} \node[partition] at (\i,0) {};
      \node[partlabel] at (0.5,0) {$[0,0]$};
      \draw (0,0) to (1,0);
   \end{tikzpicture}%
}
\newsavebox{\Pdone}
\sbox{\Pdone}{
\begin{tikzpicture}[part]
\foreach \i in {0,1,2} \node[partition] at (\i/2, 0){};
\node[partlabel] at (0.5,0) {$[0,0]$};
\end{tikzpicture}
}
\newsavebox{\Pdtwo}
\sbox{\Pdtwo}{
\begin{tikzpicture}[part]
\foreach \i in {0,1,2} \node[partition] at (\i/2, 0){};
\draw(1,0) to (1.3, 0.3);
\node[partlabel] at (0.5, 0) {$[1,0]$};
\end{tikzpicture}
}
\newsavebox{\Pdthree}
\sbox{\Pdthree}{
\begin{tikzpicture}[part]
\foreach \i in {0,1,2} \node[partition] at (\i/2, 0){};
\draw(0.5, 0) to (1.3, 0.3);
\node[partlabel] at (0.5, 0){$[1,0]$};
\end{tikzpicture}
}
\newsavebox{\Pdfour}
\sbox{\Pdfour}{
\begin{tikzpicture}[part]
\foreach \i in {0,1,2} \node[partition] at (\i/2, 0){};
\draw(0.5, 0) to (1.3, 0.4);
\draw(1,0) to (1.3, 0.25);
\node[partlabel] at (0.5, 0){$[2,0]$};
\end{tikzpicture}
}
\newsavebox{\Pdfive}
\sbox{\Pdfive}{
\begin{tikzpicture}[part]
\foreach \i in {0,1,2} \node[partition] at (\i/2, 0){};
\draw(1,0) to (1.3, 0.25);
\draw(0.5, 0) to (1,0);
\node[partlabel] at (0.5, 0){$[1,0]$};
\end{tikzpicture}
}
\newsavebox{\Pdsix}
\sbox{\Pdsix}{
\begin{tikzpicture}[part]
\foreach \i in {0,1,2} \node[partition] at (\i/2, 0){};
\draw(0.5, 0) to (1,0);
\node[partlabel] at (0.5, 0){$[0,0]$};
\end{tikzpicture}
}
\newsavebox{\Pdseven}
\sbox{\Pdseven}{
\begin{tikzpicture}[part]
\foreach \i in {0,1,2} \node[partition] at (\i/2, 0){};
\draw(0,0) to (1.3, 0.3);
\node[partlabel] at (0.5, 0){$[1,0]$};
\end{tikzpicture}
}
\newsavebox{\Pdeight}
\sbox{\Pdeight}{
\begin{tikzpicture}[part]
\foreach \i in {0,1,2} \node[partition] at (\i/2, 0){};
\draw(0,0) to (1.3, 0.4);
\draw(1,0) to (1.3, 0.25);
\node[partlabel] at (0.5, 0){$[2,0]$};
\end{tikzpicture}
}
\newsavebox{\Pdnine}
\sbox{\Pdnine}{
\begin{tikzpicture}[part]
\foreach \i in {0,1,2} \node[partition] at (\i/2, 0){};
\draw(1,0) to (1.3, 0.3);
\draw(0,0) to (1,0);
\node[partlabel] at (0.5, 0){$[1,0]$};
\end{tikzpicture}
}
\newsavebox{\Pdten}
\sbox{\Pdten}{
\begin{tikzpicture}[part]
\foreach \i in {0,1,2} \node[partition] at (\i/2, 0){};
\draw(0,0) to (1,0);
\node[partlabel] at (0.5, 0){$[0,0]$};
\end{tikzpicture}
}
\newsavebox{\Pdeleven}
\sbox{\Pdeleven}{
\begin{tikzpicture}[part]
\foreach \i in {0,1,2} \node[partition] at (\i/2, 0){};
\draw(0,0) to (1.3, 0.4);
\draw(0.5, 0) to (1.3, 0.25);
\node[partlabel] at (0.5, 0){$[2,0]$};
\end{tikzpicture}
}
\newsavebox{\Pdtwelve}
\sbox{\Pdtwelve}{
\begin{tikzpicture}[part]
\foreach \i in {0,1,2} \node[partition] at (\i/2, 0){};
\draw(0,0) to (1.3, 0.5);
\draw(0.5, 0) to (1.3, 0.35);
\draw(1,0) to (1.3, 0.2);
\node[partlabel] at (0.5, 0){$[3,0]$};
\end{tikzpicture}
}
\newsavebox{\Pdthirteen}
\sbox{\Pdthirteen}{
\begin{tikzpicture}[part]
\foreach \i in {0,1,2} \node[partition] at (\i/2, 0){};
\draw(0,0) to (1.3, 0.4);
\draw(1, 0) to (1.3, 0.25);
\draw(0.5, 0) to (1,0);
\node[partlabel] at (0.5, 0){$[2,1]$};
\end{tikzpicture}
}
\newsavebox{\Pdfourteen}
\sbox{\Pdfourteen}{
\begin{tikzpicture}[part]
\foreach \i in {0,1,2} \node[partition] at (\i/2, 0){};
\draw(0,0) to (1.3, 0.3);
\draw(0.5,0) to (1,0);
\node[partlabel] at (0.5, 0){$[1,1]$};
\end{tikzpicture}
}
\newsavebox{\Pdfifteen}
\sbox{\Pdfifteen}{
\begin{tikzpicture}[part]
\foreach \i in {0,1,2} \node[partition] at (\i/2, 0){};
\draw(0.5, 0) to (1.3, 0.4);
\draw(1,0) to (1.3, 0.25);
\draw(0,0) to (1,0);
\node[partlabel] at (0.5, 0){$[2,0]$};
\end{tikzpicture}
}
\newsavebox{\Pdsixteen}
\sbox{\Pdsixteen}{
\begin{tikzpicture}[part]
\foreach \i in {0,1,2} \node[partition] at (\i/2, 0){};
\draw(0.5, 0) to (1.3, 0.3);
\draw(0,0) to (1,0);
\node[partlabel] at (0.5, 0){$[1,0]$};
\end{tikzpicture}
}
\newsavebox{\Pdseventeen}
\sbox{\Pdseventeen}{
\begin{tikzpicture}[part]
\foreach \i in {0,1,2} \node[partition] at (\i/2, 0){};
\draw(0,0) to (0.5, 0);
\draw(0.5,0) to (1.3, 0.3);
\node[partlabel] at (0.5, 0){$[1,0]$};
\end{tikzpicture}
}
\newsavebox{\Pdeighteen}
\sbox{\Pdeighteen}{
\begin{tikzpicture}[part]
\foreach \i in {0,1,2} \node[partition] at (\i/2, 0){};
\draw(0,0) to (0.5, 0);
\draw(0.5, 0) to (1.3, 0.4);
\draw(1,0) to (1.3, 0.25);
\node[partlabel] at (0.5, 0){$[2,0]$};
\end{tikzpicture}
}
\newsavebox{\Pdnineteen}
\sbox{\Pdnineteen}{
\begin{tikzpicture}[part]
\foreach \i in {0,1,2} \node[partition] at (\i/2, 0){};
\draw(0,0) to (0.5, 0);
\draw(0.5, 0) to (1,0);
\draw(1,0) to (1.3, 0.3);
\node[partlabel] at (0.5, 0){$[1,0]$};
\end{tikzpicture}
}
\newsavebox{\Pdtwenty}
\sbox{\Pdtwenty}{
\begin{tikzpicture}[part]
\foreach \i in {0,1,2} \node[partition] at (\i/2, 0){};
\draw(0,0) to (0.5, 0);
\draw(0.5, 0) to (1,0);
\node[partlabel] at (0.5, 0){$[0,0]$};
\end{tikzpicture}
}
\newsavebox{\Pdtwentyone}
\sbox{\Pdtwentyone}{
\begin{tikzpicture}[part]
\foreach \i in {0,1,2} \node[partition] at (\i/2, 0){};
\draw(0,0) to (0.5, 0);
\node[partlabel] at (0.5, 0){$[0,0]$};
\end{tikzpicture}
}
\newsavebox{\Pdtwentytwo}
\sbox{\Pdtwentytwo}{
\begin{tikzpicture}[part]
\foreach \i in {0,1,2} \node[partition] at (\i/2, 0){};
\draw(0,0) to (0.5, 0);
\draw(1,0) to (1.3, 0.3);
\node[partlabel] at (0.5, 0){$[1,0]$};
\end{tikzpicture}
}

%
%
\begin{tikzpicture}
   \tikzset{node distance=1 and 1}%
   \node (a1) {\usebox\Paone};
   \node [below left=of a1] (b1) {\usebox\Pbone};
   \tikzset{node distance=2 and 3}%
   \node [base right=of b1] (b2) {\usebox\Pbtwo};
   \foreach \i in {b1, b2}
      \path (a1) edge (\i);
   \tikzset{node distance=1.5 and 1}%
   \node [below left=of b1] (c1) {\usebox\Pcone};
   \node [base right=of c1] (c2) {\usebox\Pctwo};
   \foreach \i in {c1, c2}
      \path (b1) edge (\i);
   \node [base right=of c2] (c3) {\usebox\Pcthree};
   \node [base right=of c3] (c4) {\usebox\Pcfour};
   \node [base right=of c4] (c5) {\usebox\Pcfive};
   \node [base right=of c5] (c6) {\usebox\Pcsix};
   \foreach \i in {c3, c4, c5, c6}
      \path (b2) edge (\i);
\foreach \i/\j in { c5/4} {
\foreach \k in {1,...,\j}
 \path (\i.base) -- +(\k/8-\j/16-1/16, -1) edge (\i);
   }
\foreach \i/\j in {c1/2, c3/4, c6/2} {
\foreach \k in {1,...,\j}
 \path (\i.base) -- +(\k/8-\j/16-1/16, -1.75) edge (\i);
   }
\tikzset{node distance=1.5 and 0.01}%
\node[below=of c1](d3){\usebox\Pdthree};
\node[base right=of d3](d4){\usebox\Pdfour};
\node[base right=of d4](d5){\usebox\Pdfive};
\node[base right=of d5](d6){\usebox\Pdsix};
\node[below left=of c4](d11){\usebox\Pdeleven};
\node[base right=of d11](d12){\usebox\Pdtwelve};
\node[base right=of d12](d13){\usebox\Pdthirteen};
\node[base right=of d13](d14){\usebox\Pdfourteen};
\node[base right=of d14](d15){\usebox\Pdfifteen};
\node[base right=of d15](d16){\usebox\Pdsixteen};
\foreach \i in {d3,d4,d5,d6}
      \path (c2) edge (\i);
\foreach \i in {d11,d12,d13,d14,d15,d16}
      \path (c4) edge (\i);
\end{tikzpicture}
\caption{Generating tree for $3$-nonnesting open partition diagrams. The first few levels agree with the generating tree for all open partition diagrams,
since the $3$-nonnesting restriction only influences levels further down the tree.}
\label{fig:parttree}
\end{figure}
\subsection{Functional equation}
\label{sec:fcteq}
We next translate the succession rule of Theorem ~\ref{thm:3part} into
generating function equations.  The generating function for
$3$-nonnesting set partitions is consequently obtained by an
evaluation of this functional equation.

As before, we define the multivariate generating
function
\[A(u,v;z)=\sum_{\pi\in \Pi^{(2)}}
u^m v^sz^{|\pi|}= \sum_{m,s,n} a_{m,s}(n) u^mv^s z^n,\]
where $[m,s]$ are the components of $\ell(\pi)$ in the first sum, and
$a_{m,s}(n)$ is the
number of $3$-nonnesting set partitions~$\pi$ at level~$n$ of the generating
tree with label $\ell(\pi)=[m,s]$. For the sake of simplicity, we
use~$A(u,v)$ interchangeably with~$A(u,v;z)$.

\begin{corollary}\label{cor:uv}
 The generating function $A(u,v)$ for $3$-nonnesting open partition diagrams, with variables $u$ and $v$ marking values $m$ and $s$ in the label, respectively,
and $z$ marking the number of vertices, satisfies the functional equation
\begin{equation}
   A(u,v)
   = 1+z \left(
      (1+u) A(u,v)
      + \Bigl( 1+\frac1{u} \Bigr) \left(
        \frac{A(u,v)}{v(1-v)} -\frac{A(uv,1)}{1-v} - \frac{A(u,0) }{v}
      \right)
   \right).
\label{eq:uv}
\end{equation}
\end{corollary}

\begin{proof}
The root~$\pi_0$ has label~$[0,0]$ and size~$n=0$. Let
$\Succ([m,s])$ be the set of labels resulting from the application of
the succession rule in Theorem~\ref{thm:3part} to the label
$[m,s]$. The generating tree gives
\begin{equation}
\label{eq:succ}
A(u,v)= 1 + \sum_{m,s,n} a_{m,s}(n) z^{n+1} \sum_{[m',s']\in\Succ([m,s])}u^{m'}v^{s'}.
\end{equation}
The terms in the interior sum originating from a fixed point and an
opener are straightforward. Let us now compute the terms coming from a transitory vertex.
These terms only appear when $m>0$, which also implies $n>0$. 
For  $s> 0$, we get
\[
\sum_{[m',s']\in \{ [m, s-1], [m,s], [m, s+1], \ldots, [m, m-1] \}}\hspace{-2cm}u^{m'}v^{s'}
= u^m (v^{s-1} + v^{s+1}+ \dots + v^{m-1}) = u^m\, \frac{v^{s-1} - v^m}{1-v}.
\]
For $s=0$, we get
\[
\sum_{[m',s']\in \{  [m,0], [m, 1], \ldots, [m, m-1] \}}\hspace{-1.5cm}u^{m'}v^{s'}
= u^m (v1 + v+ \dots + v^{m-1}) = u^m\, \frac{1 - v^m}{1-v}.
\]
Thus, the contribution in Equation~\eqref{eq:succ} from the children obtained by adding a transitory vertex is
\begin{multline*}
\sum_{n,m,s>0} a_{m,s}(n) z^{n+1}  u^m\,
\frac{v^{s-1} - v^m}{1-v} +\sum_{n,m> 0} a_{m,0}(n) z^{n+1} u^m\,
\frac{1 - v^m}{1-v}\\
=\frac{z}{1-v}\left( \sum_{n,m,s>0} a_{m,s}(n) z^n
u^mv^{s-1} +\sum_{n, m> 0} a_{m,0}(n) z^n
u^m-\sum_{n, m> 0, s\geq 0} a_{m,s}(n) z^n  u^mv^m\right).
\end{multline*}

Writing these summations in terms of evaluations of the generating function $A(u,v)$, this expression becomes
$$
\frac{z}{1-v}\left(\frac{A(u,v)-A(u,0)}{v}+A(u,0)-A(0,0)
  -(A(uv,1)-A(0,0))\right)\\
=z\left(\frac{A(u,v)}{v(1-v)}-\frac{A(uv,1)}{1-v}-\frac{A(u,0)}{v}\right).
$$
The computations for the terms coming from a closer vertex are very similar, the only difference in the formula being a factor of
$1/u$.

Combining the contributions for the four types of vertices we obtain the desired functional equation.
\end{proof}

\section{$k$-nonnesting set partitions}
\label{sec:setpartitionsk}
We follow the same general procedure to enumerate $k$-nonnesting
set partitions for general~$k$. In the first two subsections below
we generalize the generating tree construction and functional equation.
Then we consider a related family of patterns, called enhanced $k$-nestings. We shall see in
Section~\ref{sec:baxter} that the series for open partition diagrams
with neither regular nor future enhanced $3$-nestings appears to be in bijection with Baxter permutations.

\subsection{Generating tree and functional equation}
\label{sec:knonnesting}
Henceforth we shift the index and consider instead $k+1$-nonnesting
partitions.  Suppose that~$\pi$ is a $k+1$-nonnesting open partition
diagram to which we add a vertex on the rightmost end. If $\pi$ has no
future $k$-nestings, we can add any type of vertex. If, on the other
hand, $\pi$ has a future $k$-nesting and the added vertex closes a
semi-arc, it must close the top arc. Otherwise, the top arc together
with the $k$ nesting below it would create a future $k+1$-nesting.

To be able to control future $k$-nestings, we need to keep track of
$j$-nestings for $j<k$.  This is done via the \emph{nesting index\/}
of a semi-arc, which we define to be the maximum $j$ such that there
is a $j$-nesting beneath it. Equivalently, the nesting index of a
semi-arc is the largest~$j$ such that the semi-arc is in a future
$j+1$-nesting.  We track the distribution of nesting indices on the
semi-arcs, updating this distribution every time we add a vertex, and
avoiding the appearance of future $k+1$-nestings.

To each $k+1$-nonnesting open partition diagram~$\pi$, we associate a
label with $k$ components~$\ell(\pi)=[s_0, \dots, s_{k-1}]$, where
$s_i$ is defined to be the number of semi-arcs with nesting index
greater than or equal to~$i$. Figure~\ref{fig:klabeldiagram} contains
an example.  In our notation,~$s_0$ is the total number of semi-arcs,
and $s_{k-1}$ is the number of semi-arcs in a future $k$-nesting.  For
$k=2$, this labelling is consistent with Section~\ref{sec:gentree3}.
Furthermore, note that by definition, $s_0\ge s_1\ge\dots\ge
s_{k-1}\ge 0$. The label of the empty partition is~$[0,0,\dots,0]$,
since it contains no semi-arcs.

\begin{figure}[htb]
\begin{tikzpicture}[scale=0.5]
 \foreach \i in {1,...,14}
     \node[pnt,label=below:$\i$] at (\i,0)(\i) {};
\node[opnt, label=right:$\quad\mathit 0$] at (15,.5)(1a){};
\node[opnt, label=right:$\quad \mathit 1$] at (15,1.25)(2a){};
\node[opnt, label=right:$\quad \mathit 1$] at (15,2)(3a){};
\node[opnt, label=right:$\quad \mathit 2$] at (15,2.75)(4a){};
\node[opnt, label=right:$\quad \mathit 3$]at (15,3.5)(5a){};
\draw(1)  to [bend left=45] (3);
\draw(3)  to [bend left=20] (5a);
\draw(4)  to [bend left=45] (11);
\draw(5)  to [bend left=45] (10);
\draw(7)  to [bend left=20] (4a);
\draw(9) to [bend left=45] (14);
\draw(8)  to [bend left=45] (9);
\draw(10)  to [bend left=20] (3a);
\draw(11)  to [bend left=20] (2a);
\draw(12) to [bend left=45] (13);
\draw (14) to [bend left=20] (1a);
\end{tikzpicture}
\caption{An open partition diagram with label $[5, 4, 2, 1]$. The
  nesting index of each semi-arc is labelled in italics.}
\label{fig:klabeldiagram}
\end{figure}
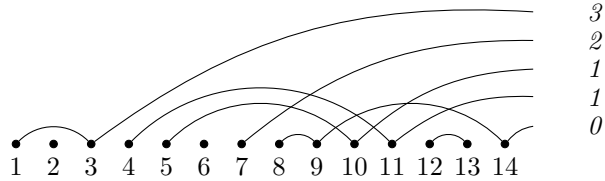

As in Section \ref{sec:gentree3}, we can predict the labels of the
children of a given node from its label alone, based on an analysis of
the four types of vertices that can be added. Since $s_{k-1}$ is the number
of semi-arcs in future $k$-nestings, the completion of such semi-arcs
must be done with care to avoid creating future $k+1$-nestings. When a
semi-arc is closed, the effect on the labels is determined by the fact
that the nesting index of those semi-arcs above it that had the same
nesting index increases by one. This is all that needs to be
tracked. These observations allow us to describe a succession rule for the
generating tree, summarized in the following theorem, which
generalizes Theorem~\ref{thm:3part}.
\begin{theorem}\label{thm:kpart}
  Let $\Pi^{(k)}$ be the set of $k+1$-nonnesting open partition
  diagrams. To each diagram, associate the label $\ell(\pi)=[s_0, \dots, s_{k-1}]$, where $s_i$ is the number of semi-arcs with nesting index $\ge i$. Then, the number of
  diagrams in $\Pi^{(k)}$ of size $n$ is the number of nodes at level
  $n$ in the generating tree with root label $[0,0,\dots,0]$, and
  succession rule given by
\begin{multline*}
 [s_0,s_1,\dots,s_{k-1}]\rightarrow\\
  \begin{array}{llr}
     ~[s_0,s_1,\dots,s_{k-1}],   & \text{\small (1)}\\[2mm]
     ~[s_0+1,s_1,\dots,s_{k-1}], & \text{\small (2)}\\[2mm]
     ~[s_0,s_1-1,\dots,s_{j-1}-1,i,s_{j+1},\dots,s_{k-1}],\quad
        \text{for $1\le j\le k-1$ and $s_j\le i\le s_{j-1}-1$,} & \text{\small (3)} \\[2mm]
     ~[s_0-1,s_1-1,\dots,s_{j-1}-1,i,s_{j+1},\dots,s_{k-1}], \quad
        \text{for $1\le j\le k-1$ and $s_j\le i\le s_{j-1}-1$,} & \text{\small (4)}\\[2mm]
     ~[s_0,s_1-1,\dots,s_{k-1}-1],
     [s_0-1,s_1-1,\dots,s_{k-1}-1], \quad
        \text{if $s_{k-1}>0$.} & \text{\small (5)}\\
\end{array}
\end{multline*}
\end{theorem}
\begin{proof}
The labels arise from adding the following kinds of vertices:
\begin{enumerate}
\item  a fixed point;
\item  a opener;
\item  a transitory;\label{eq:tran}
\item  a closer;
\item  a transitory or a closer that closes the top semi-arc, if
  the parent diagram has a future $k$-nesting.
\end{enumerate}
\end{proof}
As in Section \ref{sec:fcteq}, the generating tree in Theorem
\ref{thm:kpart} can be translated to a functional equation. Consider
the generating function $Q(v_0, v_1, \ldots, v_{k-1};z)=\sum
Q_{s_0,s_1, \ldots, s_{k-1}}(n)v_0^{s_0}v_1^{s_1}\ldots
v_{k-1}^{s_{k-1}} z^n$, where $Q_{s_0, s_1, \ldots, s_{k-1}}(n)$ is
the number of $k+1$-nonnesting open partition diagrams at level $n$ of
the generating tree with label $[s_0, s_1, \ldots, s_{k-1}]$. For
simplicity, we use the notation $Q=Q(v_0, \ldots, v_{k-1})=Q(v_0,
\ldots, v_{k-1};z)$ and $Q_{\mathbf{s}}(n)=Q_{s_0, \ldots,
  s_{k-1}}(n)$.

The equations for the addition of an opener or a fixed point are
analogous to the ones described in Section~\ref{sec:fcteq}. Now
consider the addition of a transitory that closes a semi-arc with
nesting index greater than or equal to $j$. This corresponds
to~line~\ref{eq:tran} in Theorem~\ref{thm:kpart}, giving
\begin{multline*}
  z\sum_{s_0,\dots, s_{k-1}} Q_{\mathbf{s}}(n) v_0^{s_0}v_1^{s_1-1}v_2^{s_2-1}\dots v_{j-1}^{s_{j-1}-1}(v_j^{s_j}+v_j^{s_j+1}+\dots+v_j^{s_{j-1}-1})v_{j+1}^{s_{j+1}}\dots v_{k-1}^{s_{k-1}} z^n\\
  =\frac{z}{v_1\dots v_{j-1}} \sum_{s_0,\dots, s_{k-1}} Q_{\mathbf{s}}(n)v_0^{s_0}\dots v_{k-1}^{s_{k-1}}\left( \frac{1-v_j^{s_{j-1}-s_j}}{1-v_j}\right) z^n\\
  =\frac{z}{v_1\dots v_{j-1}(1-v_j)} \left(Q-Q(v_0, \dots, v_{j-2},
    v_{j-1}v_j, 1, v_{j+1}, \dots, v_{k-1})\right).
\end{multline*}
Summing over all possible labels gives
\[\sum_{j=1}^{k-1}\frac{z}{v_1\dots v_{j-1}(1-v_j)}\left(Q-Q(v_0, \dots, v_{j-2}, v_{j-1}v_j,1,v_{j+1}, \dots, v_{k-2}, v_{k-1})\right).\]
If~$s_k>0$, the addition of a transitory corresponds to $(5)$ in
Theorem~\ref{thm:kpart}. This translates into the equation
\[
z\sum_{s_0, s_1, \dots, s_{k-2}, s_{k-1}\geq1}
Q_{\mathbf{s}}(n)v_0^{s_0}v_1^{s_1-1} \dots v_{k-1}^{s_{k-1}-1}z^n
=\frac{z}{v_1v_2\dots v_{k-1}}(Q-Q(v_0, v_1, \dots, v_{k-2}, 0)).\]
The addition of closers proceeds similarly, and combining the
expressions for all four types of vertices, we get the following
functional equation.

\begin{corollary}\label{cor:kpart}
The generating function for $k+1$-nonnesting open partition diagrams, with variable $v_i$ marking value $s_i$ in the label
and variable $z$ marking number of vertices,
denoted $Q=Q(v_0,v_1,\dots,v_{k-1})=Q(v_0,v_1,\dots,v_{k-1};z)$, satisfies the
functional equation
\begin{multline*}
   Q = 1+z (1+v_0) \biggl(
      Q + \frac{1}{v_0v_1\dots v_{k-1}} (
         Q - Q(v_0,v_1,\dots,v_{k-2},0)
      )
\\
      + \sum_{j=1}^{k-1}
         \frac{1}{v_0v_1\dots v_{j-1}(1-v_j)} (
            Q - Q(v_0,\dots,v_{j-2},v_{j-1}v_j,1,v_{j+1},\dots,v_{k-2}, v_{k-1})
         )
   \biggr).
\end{multline*}
\end{corollary}

Note that, as $k$ increases, the number of catalytic variables increases as well.

\subsection{How to use the equations}
The evaluation $v_0=0$ in $Q$ corresponds to the generating function for
$k+1$-nonnesting partitions, $Q(0,v_1,\dots,v_{k-1};z)$.  Note that
this is a function of $z$ only, because of the restriction $s_0\ge
s_1\ge\dots\ge s_{k-1}\ge 0$.  The functional equation in
Corollary~\ref{cor:kpart}, which specializes to Corollary~\ref{cor:uv}
for $k=2$, is amenable to series generation. This is explained in Section~\ref{sec:enumerative}, which also
summarizes some numerical results.

In attempting to solve these equations, the kernel method is a natural tool, but we have encountered some obstacles.
The case $k=1$ has no geometric series terms, and the equations are easily
solved for the bivariate generating function, for example, using the
orbit sum method. However, this is not surprising as this case gives simply a
generalized Catalan generating function. For higher values of $k$, the
particular subset of evaluations suggested that the symmetrizing
approach of~\cite{Bous11} might be the appropriate variant. However,
it is not directly applicable here because of a lack of a certain kind
of symmetry in the distribution of the parameter. This is an
intriguing area of future research.

\subsection{Enhanced nestings}\label{sec:enhanced}
Another relevant pattern in partition diagrams is the \emph{enhanced $k$-nesting}. An enhanced $k$-nesting is either a $k$-nesting, or a set of $k-1$ arcs $(i_1, j_1),\dots,(i_{k-1}, j_{k-1})$ and a fixed point $i_k$ (a singleton
block in the corresponding partition) such that
\[
   i_1<i_2<\dots<i_{k-1}<i_k<j_{k-1}<\dots<j_1,
\]
that is, a $k-1$-nesting with a fixed point inside the innermost arc. With a comparable definition for enhanced $k$-crossings,
the symmetry between nesting and crossing patterns for different structures has been shown to hold in
the enhanced version of the patterns as well. Set partitions are one such structure.
Bousquet-M\'elou and Xin~\cite{BoXi05} considered both enhanced and usual crossings in set partitions.
Our construction that generates $k$-nonnesting partitions can be easily
modified to generate partitions with no enhanced $k$-nesting.
For this purpose, we define a \emph{future enhanced $k$-nesting} as an enhanced $k-1$-nesting together with a semi-arc
beginning to its left, and we let the \emph{enhanced nesting index} of a semi-arc be the largest~$j$ such that it is in a future enhanced $j+1$-nesting.
The labels of the nodes in the generating tree now keep track of semi-arcs according to their enhanced nesting index.
The only difference in the construction of the tree is that the addition of a fixed point to an open partition diagram $\pi$ can create future enhanced $2$-nestings.
Indeed, the semi-arcs that had enhanced nesting index $0$ in $\pi$ have enhanced nesting index 1 after the fixed point is added. This results in the following variation of Theorem~\ref{thm:kpart}.

\begin{theorem}\label{thm:kepart}
Let $\widetilde{\Pi}^{(k)}$ be the set of open partition diagrams with neither enhanced $k+1$-nestings nor future enhanced $k+1$-nestings.
To each diagram, associate the label $\ell(\pi)=[s_0, \dots, s_{k-1}]$, where $s_i$ is the
  number of semi-arcs with enhanced nesting index $\ge i$. Then the number of
  diagrams in $\widetilde{\Pi}^{(k)}$ of size $n$ is the number of nodes at level
  $n$ in the generating tree with root label $[0,0,\dots,0]$, and
  succession rule given by
\begin{multline*}
 [s_0,s_1,\dots,s_{k-1}]\rightarrow\\
  \begin{array}{llr}
     ~[s_0,s_0,s_2,\dots,s_{k-1}],   & \text{\small (1)}\\[2mm]
     ~[s_0+1,s_1,\dots,s_{k-1}], & \text{\small (2)}\\[2mm]
     ~[s_0,s_1-1,\dots,s_{j-1}-1,i,s_{j+1},\dots,s_{k-1}],\quad
        \text{for $1\le j\le k-1$ and $s_j\le i\le s_{j-1}-1$,} & \text{\small (3)} \\[2mm]
     ~[s_0-1,s_1-1,\dots,s_{j-1}-1,i,s_{j+1},\dots,s_{k-1}], \quad
        \text{for $1\le j\le k-1$ and $s_j\le i\le s_{j-1}-1$,}& \text{\small (4)}\\[2mm]
     ~[s_0,s_1-1,\dots,s_{k-1}-1],
     [s_0-1,s_1-1,\dots,s_{k-1}-1], \quad
        \text{if $s_{k-1}>0$.} & \text{\small (5)}\\
\end{array}
\end{multline*}
\end{theorem}
\begin{proof}
The label for the addition of a fixed point, which is line (1), follows directly from the paragraph before this theorem. The labels for adding other types of vertices are obtained using the same arguments as in Theorem \ref{thm:kpart}.
\end{proof}

This succession rule can also be translated into a functional equation.
The translation is very similar to that described in Section \ref{sec:knonnesting}, the only difference being that adding a fixed point makes all semi-arcs have nesting index  at least $1$.
\begin{corollary}\label{cor:kepart}
The generating function for open partition diagrams with neither regular nor future enhanced $k+1$-nestings, with variable $v_i$ marking value $s_i$ in the label
and variable $z$ marking the number of vertices,
denoted by $P=P(v_0,v_1,\dots,v_{k-1})=P(v_0,v_1,\dots,v_{k-1};z)$, satisfies the
functional equation
\begin{multline*}
P=1+z\left(v_0P+\frac{1+v_0}{v_0v_1\dots v_{k-1}}\left(P-P(v_0,v_1,\dots,v_{k-2},0)\right)\right.\\
+\sum_{j=2}^{k-1} \frac{1+v_0}{v_0v_1\dots v_{j-1}(1-v_j)}\left(P-P(v_0,\dots,v_{j-2},v_{j-1}v_j,1,v_{j+1},\dots,v_{k-1})\right)\\
\left.+\frac{(1+v_0)P-(1+v_0v_1)P(v_0v_1,1,v_2,\dots,v_{k-1})}{v_0(1-v_1)}\right).
\end{multline*}
\end{corollary}

Again, $P(0,v_1,\dots,v_{k-1};z)$, which is a function of $z$ only, is the
generating function for partitions avoiding enhanced
$k+1$-nestings.

\subsection{Enumerative data}\label{sec:enumerative}
Corollaries~\ref{cor:kpart} and~\ref{cor:kepart} allow us to generate
data for the number of set partitions of size $n$ avoiding
$k+1$-nestings and avoiding enhanced $k+1$-nestings for small $k$ and
$n$.

We iterate the functional equations to get series
information. Specifically, we view each equation $F=1+z\Phi(F)$ as the
system $F^{[n]}=1+z\Phi(F^{[n-1]})$. Upon setting $F^{[0]}=1$, we
iterate to get successive terms in the series expansion. After $n$
iterations, we obtain the correct coefficients for $z^i$ for $0\le i\le n$, since the
functional equation is of the form $F= 1+z\Phi(F)$, where $\Phi$ is
linear in $F$ and its evaluations. Setting the catalytic variables
(i.e., those other than $z$) to $0$ results in the univariate generating series for
complete diagrams.

Tables~\ref{tab:knonnestingpartitions} and~\ref{tab:enhanced}
present the initial counting sequences, and relevant references to the
On-line Encyclopedia of Integer Sequences~\cite{oeis} for
completeness. Note that the first few terms (presented in gray)
coincide with the Bell numbers, as nesting conditions require a
minimum number of vertices. We are able to generate many more terms than
listed here. For $3$-, $4$- and $5$-nonnesting set partitions, we have
generated terms for $n$ up to $420$, $276$ and $129$, respectively.

\begin{table}[htb]
\small
\begin{tabular}{ccl}
$k+1$ & OEIS & Initial terms \\
3 &A108304 & \bt{l} \textcolor{gray}{1, 2, 5, 15, 52,} 202, 859, 3930, 19095, 97566, 520257, 2877834, 16434105, 96505490, 580864901,\\
3573876308, 22426075431, 143242527870, 929759705415, 6123822269373, 40877248201308  \et \\
4 &A108305 & \bt{l} \textcolor{gray}{1, 2, 5, 15, 52, 203, 877,} 4139, 21119, 115495, 671969, 4132936, 26723063, 180775027, 1274056792,\\
9320514343, 70548979894, 550945607475, 4427978077331, 36544023687590, 309088822019071 \et \\
5 &A192126 & \bt{l} \textcolor{gray}{1, 2, 5, 15, 52, 203, 877, 4140, 21147,} 115974, 678530, 4212654, 27627153, 190624976, 1378972826, \\
 10425400681, 82139435907, 672674215928, 5712423473216, 50193986895328, 455436027242590 \et \\
6 & A192127& \bt{l} \textcolor{gray}{1, 2, 5, 15, 52, 203, 877, 4140, 21147, 115975, 678570,} 4213596, 27644383, 190897649, 1382919174, \\
10479355676, 82850735298, 681840170501, 5828967784989, 51665915664913, 473990899143781 \et \\
7 &A192128 & \bt{l} \textcolor{gray}{1, 2, 5, 15, 52, 203, 877, 4140, 21147, 115975, 678570, 4213597, 27644437,} 190899321, 1382958475, \\
10480139391, 82864788832, 682074818390, 5832698911490, 51723290618772, 474853429890994 \et
\end{tabular}
\smallskip
\caption{Counting sequences for $k+1$-nonnesting set partitions.}
\label{tab:knonnestingpartitions}
\end{table}

\begin{table}[htb]
\small
\begin{tabular}{ccl}
$k+1$ & OEIS & Initial terms \\
3 &A108307 & \bt{l}\textcolor{gray}{1, 2, 5, 15,} 51, 191, 772, 3320, 15032, 71084, 348889, 1768483, 9220655, 49286863, 269346822,\\
1501400222, 8519796094, 49133373040, 287544553912, 1705548000296, 10241669069576 \et \\
4 &A192855& \bt{l} \textcolor{gray}{1, 2, 5, 15, 52, 203,} 876, 4120, 20883, 113034, 648410, 3917021, 24785452, 163525976, 1120523114,\\
7947399981, 58172358642, 438300848329, 3391585460591, 26898763482122, 218263920521938 \et \\
5 &A192865& \bt{l} \textcolor{gray}{1, 2, 5, 15, 52, 203, 877, 4140,} 21146, 115945, 678012, 4205209, 27531954, 189486817, 1365888674,\\
 10278272450, 80503198320, 654544093035, 5511256984436, 47950929125540, 430240226306346 \et \\
6 & A192866& \bt{l} \textcolor{gray}{1, 2, 5, 15, 52, 203, 877, 4140, 21147, 115975,} 678569, 4213555, 27643388, 190878823, 1382610179,\\
 10474709625, 82784673008, 680933897225, 5816811952612, 51505026270176, 471875801114626 \et \\
7 &A192867 & \bt{l} \textcolor{gray}{1, 2, 5, 15, 52, 203, 877, 4140, 21147, 115975, 678570, 4213597}, 27644436, 190899266, 1382956734,\\
10480097431, 82863928963, 682058946982, 5832425824171, 51718812364549, 474782378367618 \et \\
\end{tabular}
\smallskip
\caption{Counting sequences for set partitions avoiding enhanced $k+1$-nestings.}
\label{tab:enhanced}
\end{table}

\subsection{Baxter Permutations}
\label{sec:baxter}
In the course of our analysis, we observed that the number of
\emph{open} partition diagrams of size $n$ with no (regular or future)
enhanced $3$-nesting equals~$B_{n+1}$, the Baxter number of
index~$n+1$ (OEIS A001181), for~$n$ at least up to $300$.  Baxter
numbers enumerate a wide variety of combinatorial objects, such as
Baxter permutations, monotone~$2$-line meanders~\cite{Bo81}, bipolar
orientations~\cite{BoBoFu10}, and triples of nonintersecting paths
\cite{FeFuNoOr11}. Based on the numerical evidence, we conjecture the
following.
\begin{conjecture}\label{conj:baxter}
  The number of open partition diagrams on~$n$ vertices with neither
  regular nor future enhanced $3$-nestings is~$B_{n+1}$, the number of
  Baxter permutations of length~$n+1$.
\end{conjecture}
It is natural to ask for a bijective correspondence between open
partition diagrams with no enhanced $3$-nestings and one of the above
sets counted by the Baxter numbers.  At least two different generating
trees for Baxter objects appear in the literature where the labels
have two components. However, our tree for $\widetilde{\Pi}^{(2)}$
described in Theorem~\ref{thm:kepart} differs from these two trees
already at the third level, which contains $6$ elements. One
difficulty in finding a bijection is that our objects lack a
fundamental symmetry that the other objects (and their generating
trees) enjoy.  This is the case, for example, for Baxter permutations,
which are those permutations $\sigma$ having no indices $i<j<k$ such
that $\sigma(j+1)<\sigma(i)<\sigma(k)<\sigma(j)$ or
$\sigma(j)<\sigma(k)<\sigma(i)<\sigma(j+1)$.

The functional equation from Corollary~\ref{cor:kepart}
for $k=2$, after some manipulation gives
\[{\frac { -uv+u{v}^{2}+z{u}^{2}v-z{u}^{2}{v}^{2}+zu+z}{v-1}}B
 \left( u,v \right) ={\frac { \left( uv+1 \right) zv}{v-1}}B \left( u
v,1 \right) -z \left( u+1 \right) B \left( u,0 \right) +uv.
\]
We have not been able to solve this equation using the kernel
method.

 The initial series development is
\begin{multline*}
B(u,v)=1+(u+1)z+(uv+u^2+2u+2)z^2+(v^2u^2+2vu^2+4uv+6u+3u^2+u^3+5)z^3\\
+(u^3v^3+7v^2u^2+2v^2u^3+16uv+11vu^2+3vu^3+20u+12u^2+4u^3+u^4+15)z^4+\ldots.
\end{multline*}
If Conjecture~\ref{conj:baxter} is true, it may reveal which subsets
of each of the above Baxter families correspond to completed diagrams,
i.e., partitions with no enhanced $3$-nesting. Another natural
followup question would be to determine what the parameters $u$ and
$v$ mark on these Baxter objects.

Finally, we wonder whether there exist generalizations of Baxter
objects which correspond to open partition diagrams avoiding enhanced
$k$-nestings. This might provide a connection between the conjectures
on the non-D-finiteness of generating functions for pattern-avoiding
permutations~\cite{NoZe96} and Conjecture~\ref{conj:nonDfin} for
$k$-nonnesting partitions for $k>3$.

\section{$k$-nonnesting permutations}
\label{sec:permutations}
There are some similarities between $k$-nonnesting permutations and
$k$-nonnesting set partitions, yet none of the techniques that have
been applied in the literature to the enumeration of $k$-nonnesting
set partitions~\cite{BoXi05,MiYe13} works for permutations. On the
other hand, our method can be modified to take advantage of these
similarities and deal with $k$-nonnesting permutations as well.  This
gives rise to the first substantial set of enumerative data on
$k$-nonnesting permutations. Previously, only data up to $n=12$ was
known, and obtaining it required extensive computation.

First, we recall how to represent a permutation as an arc
diagram. This representation was first used by
Corteel~\cite{Cort07}, and in a modified form by
Elizalde~\cite{Eliz11}. It is essentially a drawing of the
cycle structure of the permutation. Given~$\sigma\in\mathfrak{S}_n$,
the diagram of~$\sigma$ has an arc between~$i$ and $\sigma(i)$ for
each~$i$ from 1 to~$n$, and the arc is drawn above the vertices (and
is called an \emph{upper arc}) if $i\leq \sigma(i)$, and below the
vertices (and is called a \emph{lower arc}) if $i>\sigma(i)$. %
We call such a representation a \emph{permutation diagram} of size $n$.
In this notation, a subset of $k$ arcs is a \emph{$k$-nesting} if either
\begin{enumerate}
\item all $k$ arcs are upper arcs and form an enhanced $k$-nesting with the definition from Section~\ref{sec:enhanced} (considering arcs of the form $(i,i)$ to be fixed points);
\item all $k$ arcs are lower arcs and form a $k$-nesting with the definition from Section~\ref{sec:intro}.
\end{enumerate}
We call the two above possibilities \emph{upper enhanced $k$-nestings}
and \emph{lower $k$-nestings}, respectively.  The reason behind the
slight dissymmetry in the definition comes from the original paper of
Corteel~\cite{Cort07}, where this was necessary for bijections between
certain classes of permutations.  Burrill \emph{et
  al.}~\cite{BuMiPo10} maintained this dissymmetry. In any case, the
construction we provide can be easily adapted to make a uniform
treatment of upper and lower arcs. For clarity, we  draw an
arc from $i$ to itself when $\sigma(i)=i$, since  such a fixed point
contributes to enhanced $k$-nestings of upper arcs.
Figure~\ref{fig:permutation} shows an example of a permutation
diagram.

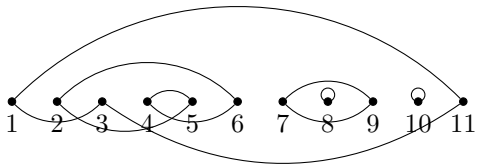
\begin{figure}[htb]
\begin{tikzpicture}[scale=0.6]
 \foreach \i in {1,...,11}
     \node[pnt,label=below:$\i$] at (\i,0)(\i) {};
\draw(1)  to [bend left=45] (11);
\draw(1)  to [bend left=315](3);
\draw(3)  to [bend left=325](11);
\draw(2)  to [bend left=45] (6);
\draw(4)  to [bend left=315] (6);
\draw(4)  to [bend left=45](5);
\draw(2)  to [bend left=315] (5);
\draw(7)  to [bend left=45] (9);
\draw(7)  to [bend left=315] (9);
\draw(8, 0.15) circle (0.15);
\draw(10, 0.15) circle (0.15);
\end{tikzpicture}
\caption{The permutation $\sigma=(1\,11\,3)(2\,6\,4\,5)(7\,9)(8)(10)$
  and its arc diagram representation. This diagram has two $3$-nestings:
  $\{(1,11) ,(2,6),(4,5)\}$ and $\{(1,11),(7,9), (8,8)\}$.}
\label{fig:permutation}
\end{figure}

As in the case of set partitions, there is a related definition
of~$k$-crossings in permutations. In~\cite{BuMiPo10}, the authors proved that
$k$-noncrossing and~$k$-nonnesting permutations are
equinumerous. However, they present very limited enumerative results,
which we significantly improve upon here.

We remark that the subset of upper arcs forms a set partition, and so
does the subset of lower arcs.  However, these two partitions are not
independent.  Their relationship is explicitly described
in~\cite{BuMiPo10}. This interpretation as a pair of partitions allows
us to extend the construction from Section~\ref{sec:setpartitionsk}
quite naturally.

The vertices of a permutation diagram can be of five types:
\emph{fixed points}, \emph{openers}, \emph{closers}, \emph{upper transitories}, and
\emph{lower transitories}. A fixed point has no incident edges, an opener (resp. closer)
is the left (resp. right) endpoint of an upper and a lower arc, and an upper (resp. lower) transitory
vertex is the right endpoint of an upper (resp. lower) arc and the left endpoint of another.

\subsection{Generating tree for open permutation diagrams}\label{sec:gentreeperm}
Similarly to what we did for partitions, we consider a more general
class of objects that we call \emph{open permutation diagrams}, by
allowing three additional vertex types: \emph{openers}, \emph{upper
  semi-transitories}, and \emph{lower semi-transitories}. An opener is
the left endpoint of an upper and a lower semi-arc, and an upper
(resp. lower) transitory vertex is the right endpoint of an upper
(resp. lower) arc and the right endpoint of an upper (resp. lower)
semi-arc. Note that in an open permutation diagram, the number of
upper semi-arcs equals the number of lower semi-arcs. This is because
when a fixed point is added, no upper or lower arcs are added; when an
opener is added, the number of upper and lower semi-arcs both increase
by $1$; when an upper transitory is added, an upper semi-arc is
added then removed, while the lower semi-arcs remain unchanged;
similarly when a lower transitory is added, the lower arcs lose and
gain a semi-arc, and upper arcs remain unchanged; and lastly when a
closer is added, both upper and lower arcs lose a semi-arc.

An open permutation diagram can also be viewed as a permutation where each cycle of length $i$ can be coloured one of $i+1$ possible colours: one to indicate that the cycle has no semi-arcs,
and the other $i$ for the possible choices of an arc in the cycles to be converted into two semi-arcs.

\begin{figure}[htb]
\begin{tikzpicture}[scale=0.6]
\foreach \x in {1,2,3}{\node[pnt, label=below:$\x$] at (\x, 0){};};
\draw[bend left=45](1,0) to (2,0);
\draw[bend left=45](2,0) to (3,0);
\draw[bend left=-45](1,0) to (3,0);
\end{tikzpicture}
\hspace{0.5cm}
\begin{tikzpicture}[scale=0.6]
\foreach \x in {1,2,3}{\node[pnt, label=below:$\x$] at (\x, 0){};};
\draw[bend left=45](1,0) to (2,0);
\draw[bend left=45](2,0) to (3,0);
\draw[bend left=45](3,0) to (3.5, 0.5);
\draw[bend left=-45](1,0) to (3.5, -0.5);
\end{tikzpicture}
\hspace{0.5cm}
\begin{tikzpicture}[scale=0.6]
\foreach \x in {1,2,3}{\node[pnt, label=below:$\x$] at (\x, 0){};};
\draw[bend left=45](1,0) to (2,0);
\draw[bend left=45](2,0) to (3.5, 0.5);
\draw[bend left=-45](1,0) to (3,0);
\draw[bend left=-45](3,0) to (3.5, -0.5);
\end{tikzpicture}
\hspace{0.5cm}
\begin{tikzpicture}[scale=0.6]
\foreach \x in {1,2,3}{\node[pnt, label=below:$\x$] at (\x, 0){};};
\draw[bend left=45](1,0) to (3.5, 0.5);
\draw[bend left=45](2,0) to (3,0);
\draw[bend left=-45](1,0) to (3,0);
\draw[bend left=-45](2,0) to (3.5, -0.5);
\end{tikzpicture}

\caption{A permutation $\sigma=(123)$ with a cycle of length $3$ may
  be coloured in 4 different ways, or have either one of the existing arcs
  or no arcs converted to a semi-arc: $(123)$, $(123*)$,  $(12*3)$, and $(1*23)$. }
\end{figure}
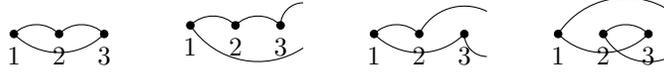

It follows that the exponential generating function for open permutation diagrams is $$\frac{1}{1-z}\exp\left(\frac{z}{1-z}\right).$$

To create an open permutation diagram of size~$n$ from one of
size~$n-1$, we add a vertex labelled~$n$ which can be of any of the five following types:
\begin{enumerate}
\item fixed point (\begin{tikzpicture}[scale=0.6] \node[pnt]at (0,0){}; \draw(0,0.15) circle (0.15);\end{tikzpicture});
\item opener (\begin{tikzpicture}\node[pnt] at (0,0){}; \draw[bend left=45](0,0) to (0.15, 0.15);\draw[bend left=-45](0,0) to (0.15, -0.15);\end{tikzpicture});
\item upper transitory (\begin{tikzpicture}\node[pnt] at (0,0){}; \draw[bend left=45](0,0) to (0.15, 0.15);\draw[bend left=45](-0.15, 0.15) to (0,0);\end{tikzpicture}) (provided there is an available upper semi-arc);
\item lower transitory (\begin{tikzpicture}\node[pnt] at (0,0){};\draw[bend left=-45](0,0) to (0.15, -0.15);\draw[bend left=-45](-0.15, -0.15) to (0,0);\end{tikzpicture}) (provided there is an available lower semi-arc);
\item closer (\begin{tikzpicture}\node[pnt] at (0,0){};\draw[bend left=45](-0.15, 0.15) to (0,0);\draw[bend left=-45](-.15, -0.15) to (0,0);\end{tikzpicture}) (provided there are available upper and lower semi-arcs).
\end{enumerate}
We can easily make a generating tree for the class of
open permutation diagrams by keeping track of the number $h$ of upper semi-arcs. The root has label $[0]$, and the succession rule is given by 

$$
[h] \rightarrow
\begin{array}{llcr}
~[h], && \quad & \text{(fixed point)} \\[2mm]
~[h+1], &&& \text{(opener)} \\[2mm]
~\underbrace{[h], [h], \dots, [h]}_{h \text{ copies}}, & \text{if $h>0$,} &&\text{(upper transitory)}\\[2mm]
~\underbrace{[h], [h], \dots, [h]}_{h \text{ copies}}, & \text{if $h>0$,} &&\text{(lower transitory)}\\[2mm]
~\underbrace{[h-1],  [h-1],\dots,   [h-1]}_{h^2 \text{ copies}}, &\text{if $h>0$.}&&\text{(closer)}\\[2mm]
\end{array}
$$

To incorporate the nesting constraint to the tree, we define a notion of future nestings in open permutation diagrams.
A \emph{future enhanced upper $k$-nesting} is an upper enhanced $k-1$-nesting together with an upper semi-arc
beginning to its left, and a \emph{future lower $k$-nesting} is a lower $k-1$-nesting together with a lower semi-arc
beginning to its left. An example of each is in Figure~\ref{fig:partial-permutation}.
The enhanced nesting index of an upper semi-arc is the largest~$j$ such that the semi-arc is in a future enhanced upper
$j+1$-nesting.
The nesting index of a lower semi-arc is the largest~$j$ such that the semi-arc is in a future lower $j+1$-nesting.
An open permutation diagram is called \emph{$k$-nonnesting} if it has no regular or future enhanced upper $k$-nesting, and no regular or future lower $k$-nesting.

\begin{figure}[htb]
\begin{tikzpicture}[scale=0.6]
\node[pnt, label=below:$1$] at (0,0)(1){};
\node[pnt,label=below:$2$] at (1,0)(2){};
\node[pnt,label=below:$3$] at (2,0)(3){};
\node[pnt,label=below:$4$] at (3,0)(4){};
\node[pnt,label=below:$5$] at (4,0)(5){};
\node[pnt,label=below:$6$] at (5,0)(6){};
\node[pnt,label=below:$7$] at (6,0)(7){};
\node[pnt,label=below:$8$] at (7,0)(8){};
\node[pnt,label=below:$9$] at (8,0)(9){};
\node[pnt,label=below:$10$] at (9,0)(10){};
\node[pnt,label=below:$11$] at (10,0)(11){};
\node[pnt,label=below:$12$] at (11,0)(12){};
\node[pnt,label=below:$13$] at (12,0)(13){};
\node[opnt,  label=right:$\quad\mathit 0$] at (13,0.5)(1a){};
\node[opnt,  label=right:$\quad\mathit 0$] at (13,1)(2a){};
\node[opnt,  label=right:$\quad\mathit 2$] at (13,1.5)(3a){};
\node[opnt, label=right:$\quad\mathit 0$] at (13,-0.5)(1b){};
\node[opnt, label=right:$\quad\mathit 0$] at (13,-1)(2b){};
\node[opnt, label=right:$\quad\mathit 1$] at (13,-1.5)(3b){};
\draw(1)  to [bend left=45] (11);
\draw(2)  to [bend left=45] (6);
\draw(5)  to [bend left=315] (6);
\draw(2)  to [bend left=315] (5);
\draw(7) to  [bend left=45] (12);
\draw(7) to  [bend left=315] (10);
\draw(8)  to [bend left=45] (9);
\draw(8)  to [bend left=315] (12);
\draw(10)  to [bend left=340] (2b);
\draw(1)  to [bend left=315] (9);
\draw(3) to [bend left=20] (3a);
\draw(11) to [bend left=20] (2a);
\draw(3) to [bend left=340] (3b);
\draw(13) to [bend left=20] (1a);
\draw(13) to [bend left=340] (1b);
\draw(3,0.15) circle (0.15);
\end{tikzpicture}
\caption{An open permutation diagram. The upper arcs $(7,12)$, $(8,9)$
  and the upper semi-arc $(3,*)$ form a future enhanced upper $3$-nesting. The lower arc $(7,10)$
  and the lower semi-arc $(3,*)$ form a future lower $2$-nesting. The nesting index of each
  semi-arc is labelled in italics.}
\label{fig:partial-permutation}
\end{figure}
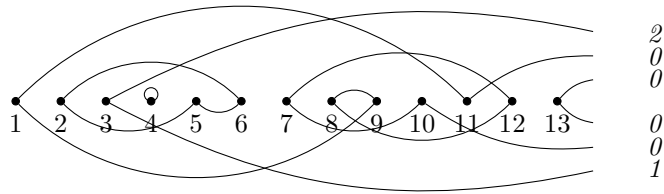

Again, it is convenient to shift the index and consider
$k+1$-nonnesting permutations from now on.  To each $k+1$-nonnesting
open permutation diagram, we associate a label consisting of a
$3$-tuple, $[h;\r;\s]$. Here, $h$ is the number of upper semi-arcs
(and hence also the number of lower semi-arcs),
$\r=[r_1,\dots,r_{k-1}]$ is a vector such that~$r_i$ is the number of
upper semi-arcs of enhanced nesting index greater than or equal
to~$i$, and $\s=[s_1,\dots,s_{k-1}]$ is a vector such that~$s_i$ is
the number of lower semi-arcs of nesting index greater than or equal
to~$i$. The label of the $4$-nonnesting diagram in
Figure~\ref{fig:partial-permutation} is $[3;1,1;1,0]$.

We note that in an open permutation diagram $\sigma$, if we consider
loops (which are cycles of length $1$) as fixed points, the upper
(resp. lower) arcs and semi-arcs form an open partition diagram
$\sigma^+$ (resp. $\sigma^-$) on the vertices $\{1,\dots,n\}$.  If the
label of $\sigma$ is $[h;\r;\s]$, then the label of $\sigma^+$ as
described in Theorem~\ref{thm:kpart} is $[h,\r]$, and the label of
$\sigma^-$ as described in Theorem~\ref{thm:kepart} is $[h,\s]$. In
particular, $h\ge r_1\ge\dots\ge r_{k-1}\ge 0$ and $h\ge
s_1\ge\dots\ge s_{k-1}\ge 0$.

\subsection{$3$-nonnesting permutations}
\label{sec:3perm}
Again, the case of $3$-nonnesting permutations is sufficiently insightful for the
general method without being overly complicated. In this section,  we describe the
generating tree for $3$-nonnesting open permutation diagrams, and we determine a functional equation for the
generating function.

\subsubsection{Generating tree}
The label of a $3$-nonnesting open permutation diagram is~$[h, r, s]$
(we use commas instead of semicolons in this section). Here, $2h$ is
the total number of semi-arcs, $r$ is the number of upper semi-arcs
that belong to a future enhanced upper $2$-nesting, and $s$ is the
number of lower semi-arcs that belong to a future lower
$2$-nesting. Notice that in this case the vectors $\r$ and $\s$ each
only have one component, thus we dispense with the formal vector
notation while working through this example. The empty diagram has
label $[0,0,0]$. An example of the labelling is given in
Figure~\ref{fig:labeldiagramperm}, where the numbers on the vertices
are omitted.  The arrows indicate semi-arcs that are in future
$2$-nestings: two upper semi-arcs and one lower semi-arc, hence the
label of the diagram is $[4, 2, 1]$.

\begin{figure}[ht]
\begin{tikzpicture}[scale=0.4]
\foreach \i in {1,...,11}
     \node[pnt] at (\i,0)(\i) {};
\node[label=right:$\quad\leftarrow$] at (11,1.5){};
\node[label=right:$\quad\leftarrow$] at (11,2){};
\node[label=right:$\quad\leftarrow$] at (11, -2){};
\draw(1)  to [bend left=45] (3) to  [bend left=20] (12, 2);
\draw(4)  to [bend left=45] (6);
\draw(5) to [bend left=45](8) to [bend left=30] (9) to [bend left=30] (10) to [bend left=20] (12,1);
\draw(7) to  [bend left=20] (12, 1.5);
\draw(11) to [bend left=20] (12, 0.5);
\draw(1)  to [bend left=-20](12, -2);
\draw(4) to [bend left=-45](6);
\draw(5) to [bend left=-20](12, -1.5);
\draw(7) to [bend left=-20](12, -1);
\draw(11) to [bend left=-20](12, -0.5);
\draw(2, 0.15) circle (0.15);
\end{tikzpicture}
\caption{{An arc diagram with label $[4, 2, 1]$.}}
\label{fig:labeldiagramperm}
\end{figure}
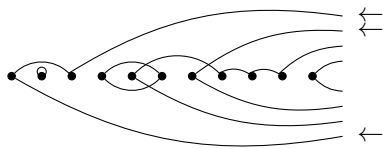

To predict the labels of the children of a $3$-nonnesting open permutation diagram, we consider the
different types of vertices that can be added. To avoid future $3$-nestings,
we are not allowed to add a closer or transitory vertex that closes a semi-arc belonging to a future enhanced upper $2$-nesting or a future lower $2$-nesting, unless it is an outermost semi-arc, i.e., the
top upper semi-arc or the bottom lower semi-arc. As an example,
the labels of the children of the diagram in Figure~\ref{fig:labeldiagramperm} are generated by adding vertices of different types as follows.
\begin{enumerate}
\item Fixed point: one child with label $[4,4,1]$, since the upper semi-arcs belong now to future enhanced upper $2$-nestings.
\item Opener: one child with label $[5,2,1]$.
\item Upper transitory: closing the upper semi-arcs that are not in future enhanced $2$-nestings gives the labels $[4,2,1]$ and $[4,3,1]$;
closing the top semi-arc (the only one in a $2$-nesting that we are allowed to close) removes one future upper $2$-nesting, giving the label $[4,1,1]$.
\item Lower transitory: all lower semi-arcs can be closed, and the four resulting labels are $[4,2,0]$, $[4,2,1]$, $[4,2,2]$ and $[4,2,3]$.
\item Closer: we simultaneously and independently close an upper and a lower semi-arc, among those that we are allowed to close. There are three choices for the
  former and four for the latter, giving twelve children with labels $[ 3,1,0]$, $[3,1,1]$, $[3,1,2]$, $[3,1,3]$, $[3, 2,0]$, $[3,2,1]$,  $[3,2,2]$, $[3,2,3]$, $[3, 3,0]$, $[3,3,1]$, $[3,3,2]$, $[3,3,3]$, that is, $\{3\} \times \{1,2,3\} \times \{0,1,2,3\}$.
\end{enumerate}
The succession rule for the generating tree is described in Theorem~\ref{thm:permgenscheme}.
Another example is drawn in Figure~\ref{fig:permchildren}, and the first few levels of the tree appear in Figure~\ref{fig:permtree}.

\begin{figure}[htb]
\center
\begin{tikzpicture}[scale=0.4]
\foreach \i in {1, 2}
     \node[partition] at (\i,0) {};
\draw[bend left=45](1,0) to (2.5,.6);
\draw[bend left=-45](1,0) to (2.5,-.6);
\draw[bend left=45](2,0) to (2.5,.4);
\draw[bend left=-45](2,0) to (2.4,-.4);
\node at (2,-2){[2,0,0]};
\end{tikzpicture}\\
{\begin{tikzpicture}[scale=0.4]
\draw(-0.5,0) to (-1, -1);
\draw(-0.4,0) to (-0.8, -1);
\draw(-0.3,0) to (-0.6, -1);
\draw(-0.2,0) to (-0.4, -1);
\draw(-0.1,0) to (-0.2, -1);
\draw(0,0) to (0.2, -1);
\draw(0.1,0) to (0.4, -1);
\draw(0.2,0) to (0.6, -1);
\draw(0.3,0) to (0.8, -1);
\draw(0.4,0) to (1, -1);
\end{tikzpicture}
\vspace{0.25cm}
}\\
\begin{tikzpicture}[scale=0.5] 
\foreach \i in {1, 2, 3}
     \node[partition] at (\i,0) {};
\draw[bend left=45](1,0) to (3.5,.6);
\draw[bend left=-45](1,0) to (3.5,-.6);
\draw[bend left=45](2,0) to (3.5,.4);
\draw[bend left=-45](2,0) to (3.5, -.4);
\node at (2,-2){[2,2,0]};
\draw(3, 0.15) circle (0.15);
\end{tikzpicture}
\hspace{0.3cm}
\begin{tikzpicture}[scale=0.5] 
\foreach \i in {1, 2, 3}
     \node[partition] at (\i,0) {};
\draw[bend left=45](1,0) to (3.5,.8);
\draw[bend left=-45](1,0) to (3.5,-.8);
\draw[bend left=45](2,0) to (3.5,.6);
\draw[bend left=-45](2,0) to (3.5, -.6);
\draw[bend left=45](3,0) to (3.5,.4);
\draw[bend left=-45](3,0) to (3.5, -.4);
\node at (2,-2){[3,0,0]};
\end{tikzpicture}
\hspace{0.3cm}
\begin{tikzpicture}[scale=0.5] 
\foreach \i in {1, 2, 3}
     \node[partition] at (\i,0) {};
\draw[bend left=45](1,0) to (3.5,.6);
\draw[bend left=-45](1,0) to (3.5,-.6);
\draw[bend left=45](2,0) to (3,0) to (3.5,.4);
\draw[bend left=-45](2,0) to (3.5, -.4);
\node at (2,-2){[2,1,0]};
\end{tikzpicture}
\hspace{0.3cm}
\begin{tikzpicture}[scale=0.5] 
\foreach \i in {1, 2, 3}
     \node[partition] at (\i,0) {};
\draw[bend left=45](1,0) to (3,0) to (3.5,.4);
\draw[bend left=-45](1,0) to (3.5,-.6);
\draw[bend left=45](2,0) to (3.5,.6);
\draw[bend left=-45](2,0) to (3.5, -.4);
\node at (2,-2){[2,0,0]};
\end{tikzpicture}
\begin{tikzpicture}[scale=0.5] 
\foreach \i in {1, 2, 3}
     \node[partition] at (\i,0) {};
\draw[bend left=45](1,0) to (3.5,.6);
\draw[bend left=-45](1,0) to (3.5,-.6);
\draw[bend left=45](2,0) to (3.5,.4);
\draw[bend left=-45](2,0) to (3,0) to (3.5, -.4);
\node at (2,-2){[2,0,1]};
\end{tikzpicture}
\hspace{0.3cm}
\begin{tikzpicture}[scale=0.5] 
\foreach \i in {1, 2, 3}
     \node[partition] at (\i,0) {};
\draw[bend left=45](1,0) to (3.5,.6);
\draw[bend left=-45](1,0) to (3,0) to (3.5,-.4);
\draw[bend left=45](2,0) to (3.5,.4);
\draw[bend left=-45](2,0) to (3.5, -.6);
\node at (2,-2){[2,0,0]};
\end{tikzpicture}

\hspace{0.3cm}
\begin{tikzpicture}[scale=0.5] 
\foreach \i in {1, 2, 3}
     \node[partition] at (\i,0) {};
\draw[bend left=45](1,0) to (3.5,.6);
\draw[bend left=-45](1,0) to (3.5,-.6);
\draw[bend left=45](2,0) to (3,0);
\draw[bend left=-45](2,0) to (3,0);
\node at (2,-2){[1,1,1]};
\end{tikzpicture}
\hspace{0.3cm}
\begin{tikzpicture}[scale=0.5] 
\foreach \i in {1, 2, 3}
     \node[partition] at (\i,0) {};
\draw[bend left=45](1,0) to (3.5,.6);
\draw[bend left=-45](1,0) to (3,0);
\draw[bend left=45](2,0) to (3,0);
\draw[bend left=-45](2,0) to (3.5,-.6);
\node at (2,-2){[1,1,0]};
\end{tikzpicture}
\hspace{0.3cm}
\begin{tikzpicture}[scale=0.5] 
\foreach \i in {1, 2, 3}
     \node[partition] at (\i,0) {};
\draw[bend left=45](1,0) to (3,0);
\draw[bend left=-45](1,0) to (3,0);
\draw[bend left=45](2,0) to (3.5,.6);
\draw[bend left=-45](2,0) to (3.5,-.6);
\node at (2,-2){[1,0,0]};
\end{tikzpicture}
\hspace{0.3cm}
\begin{tikzpicture}[scale=0.5] 
\foreach \i in {1, 2, 3}
     \node[partition] at (\i,0) {};
\draw[bend left=45](1,0) to (3,0);
\draw[bend left=-45](1,0) to (3.5,-.6);
\draw[bend left=45](2,0) to (3.5,.6);
\draw[bend left=-45](2,0) to (3,0);
\node at (2,-2){[1,0,1]};
\end{tikzpicture}
\caption{A $3$-nonnesting open permutation diagram and its children.}
\label{fig:permchildren}
\end{figure}
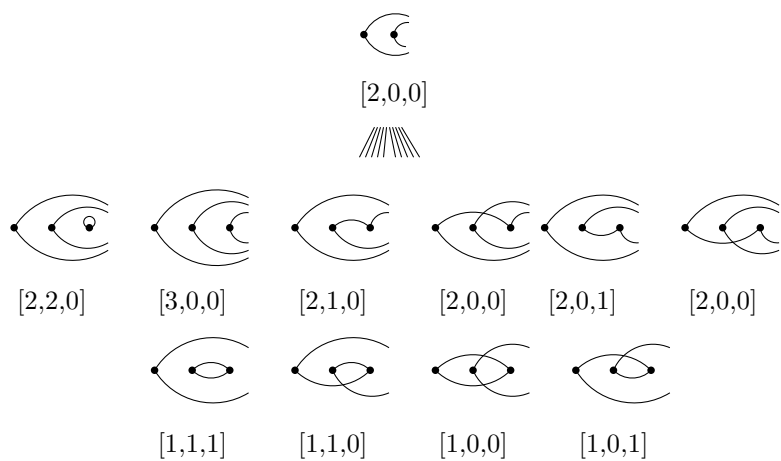

\begin{theorem}\label{thm:permgenscheme}
  Let~$\Sigma^{(2)}$ be the set of $3$-nonnesting open permutation
  diagrams. To each diagram~$\sigma$, associate the label
  $\ell(\sigma)=[h,r,s]$, where $2h$ is the total number of semi-arcs,
  $r$ is the number of semi-arcs in a future enhanced upper $2$-nesting
  and $s$ is the number of semi-arcs in a future lower
  $2$-nesting. Then, the number of diagrams in $\Sigma^{(2)}$ of size
  $n$ is the number of nodes at level $n$ in the generating tree with
  root label $[0,0,0]$, and succession rule given by
\begin{equation*}
[h, r, s] \rightarrow
\begin{array}{llr}
~[h, h, s], &&\text{\small (1)} \\[2mm]
~[h+1,r,s],&&\text{\small (2)}  \\[2mm]
~[h, i, s],& \text{for $\max\{0,r-1\}\le i\le h-1$,} &\text{\small (3)}\\[2mm]
~[h, r, j],& \text{for $\max\{0,s-1\}\le j\le h-1$,} &\text{\small (4)}\\[2mm]
~[h-1, i, j],&\text{for $\max\{0,r-1\}\le i\le h-1$ and $\max\{0,s-1\}\le j\le h-1$.} & \text{\small (5)}\\
\end{array}
\end{equation*}
The number of $3$-nonnesting permutations of size $n$ is equal to the
number of nodes with label $[0,0,0]$ at the $n$-th level of this generating
tree.
\end{theorem}

\begin{proof}The labels correspond to the addition of the following vertices to a diagram of $\sigma$:
\begin{enumerate}
\item A fixed point, which results in all the upper semi-arcs becoming part of future enhanced $2$-nestings;
\item an opener, which produces a new upper semi-arc and a new lower one, neither of which is in a future $2$-nesting;
\item an upper transitory closing a semi-arc not belonging to a future enhanced upper $2$-nesting or, if $r>0$, possibly closing the top semi-arc;
\item a lower transitory closing a semi-arc not belonging to a future lower $2$-nesting or, if $s>0$, possibly closing the bottom semi-arc;
\item a closer, which can close any combination of an upper and a lower semi-arc among those allowed to close in parts~(3) and~(4).
\end{enumerate}
\end{proof}

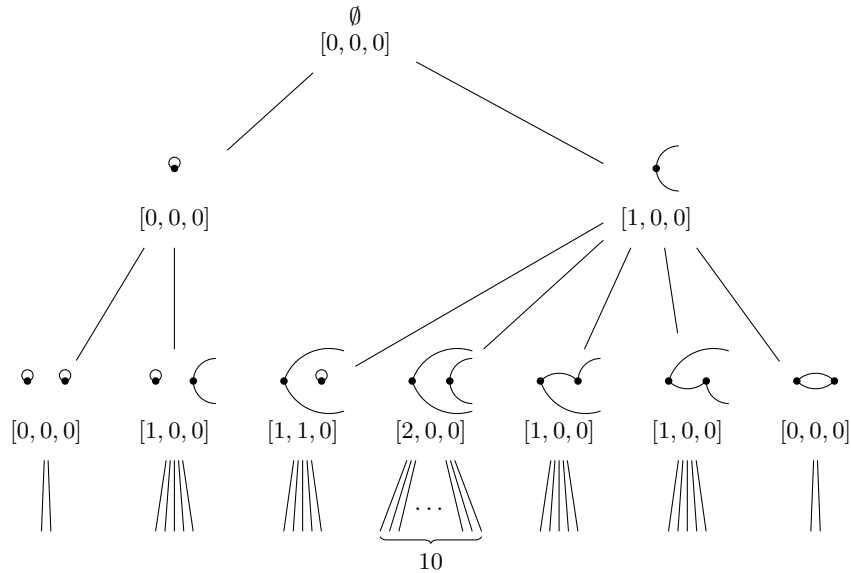
\begin{figure}[htb]
\begin{center}\small
\tikzset{partlabel/.style={below=0.4}}
\sbox{\Paone}{%
   \begin{tikzpicture}[part]
      \node {$\emptyset$};
      \node[partlabel,below=3pt] {$[0,0,0]$};
   \end{tikzpicture}
}
\sbox{\Pbone}{%
   \begin{tikzpicture}[part]
      \node[partition] {};
      \node[partlabel] {$[0,0,0]$};
      \draw(0, 0.15) circle (0.15);
   \end{tikzpicture}%
}
\sbox{\Pbtwo}{%
   \begin{tikzpicture}[part]%
      \node[partition] at (0,0) {};
      \draw (0,0) to (0.6,0.6);
      \draw[bend right] (0,0) to (0.6,-0.6);
      \node[partlabel] {$[1,0,0]$};

   \end{tikzpicture}%
}
\sbox{\Pcone}{%
   \begin{tikzpicture}[part]%
      \foreach \i in {0,1} \node[partition] at (\i,0) {};
      \node[partlabel] at (0.5,0) {$[0,0,0]$};
   \draw(0, 0.15) circle (0.15);
     \draw(1, 0.15) circle (0.15);
   \end{tikzpicture}%
}
\sbox{\Pctwo}{%
   \begin{tikzpicture}[part]%
      \foreach \i in {0,1} \node[partition] at (\i,0) {};
      \node[partlabel] at (0.5,0) {$[1,0,0]$};
      \draw (1,0) to (1.6,0.6);
      \draw[bend right] (1,0) to (1.6,-0.6);
\draw(0, 0.15) circle (0.15);
   \end{tikzpicture}%
}
\sbox{\Pcthree}{%
   \begin{tikzpicture}[part]%
      \foreach \i in {0,1} \node[partition] at (\i,0) {};
      \node[partlabel] at (0.5,0) {$[1,1,0]$};
      \draw (0,0) to (1.6, 0.8);
      \draw[bend right] (0,0) to (1.6,-0.8);
   \draw(1, 0.15) circle (0.15);
   \end{tikzpicture}%
}
\sbox{\Pcfour}{%
   \begin{tikzpicture}[part]%
      \foreach \i in {0,1} \node[partition] at (\i,0) {};
      \node[partlabel] at (0.5,0) {$[2,0,0]$};
      \draw (0,0) to (1.6, 0.8);
      \draw[bend right] (0,0) to (1.6,-0.8);
      \draw (1,0) to (1.6, 0.6);
      \draw[bend right] (1,0) to (1.6,-0.6);
   \end{tikzpicture}%
}
\sbox{\Pcfive}{%
   \begin{tikzpicture}[part]%
      \foreach \i in {0,1} \node[partition] at (\i,0) {};
      \node[partlabel] at (0.5,0) {$[1,0,0]$};
      \draw (0,0) to (1,0);
      \draw (1,0) to (1.6, 0.6);
      \draw[bend right] (0,0) to (1.6, -0.8);
   \end{tikzpicture}%
}
\sbox{\Pcsix}{%
   \begin{tikzpicture}[part]%
      \foreach \i in {0,1} \node[partition] at (\i,0) {};
      \node[partlabel] at (0.5,0) {$[1,0,0]$};
      \draw[bend right] (0,0) to (1,0);
      \draw (0,0) to (1.6, 0.8);
      \draw[bend right] (1,0) to (1.6,-0.6);
   \end{tikzpicture}%
}
\newsavebox{\Pcseven}%
\sbox{\Pcseven}{%
   \begin{tikzpicture}[part]%
      \foreach \i in {0,1} \node[partition] at (\i,0) {};
      \node[partlabel] at (0.5,0) {$[0,0,0]$};
      \draw (0,0) to (1,0);
      \draw[bend right] (0,0) to (1,0);
   \end{tikzpicture}%
}
%
%
\begin{tikzpicture}
   \tikzset{node distance=1 and 1}%
   \node (a1) {\usebox\Paone};
   \node [below left=of a1] (b1) {\usebox\Pbone};
   \tikzset{node distance=2 and 5}%
   \node [base right=of b1] (b2) {\usebox\Pbtwo};
   \foreach \i in {b1, b2}
      \path (a1) edge (\i);
   \tikzset{node distance=1.5 and 0.3}%
   \node [below left=of b1] (c1) {\usebox\Pcone};
   \node [base right=of c1] (c2) {\usebox\Pctwo};
   \foreach \i in {c1, c2}
      \path (b1) edge (\i);
   \node [base right=of c2] (c3) {\usebox\Pcthree};
   \node [base right=of c3] (c4) {\usebox\Pcfour};
   \node [base right=of c4] (c5) {\usebox\Pcfive};
   \node [base right=of c5] (c6) {\usebox\Pcsix};
   \node [base right=of c6] (c7) {\usebox\Pcseven};
   \foreach \i in {c3, c4, c5, c6, c7}
      \path (b2) edge (\i);
   \foreach \i/\j in {c1/2, c2/5, c3/5, c5/5, c6/5, c7/2} {
      \foreach \k in {1,...,\j}
         \path (\i.base) -- +(\k/8-\j/16-1/16, -2) edge (\i);
   }
   \foreach \k in {1,2,3} {
      \path (c4.base) -- +(-\k/8-0.3,-2) edge (c4);
      \path (c4.base) -- +( \k/8+0.3,-2) edge (c4);
   }
   \path (c4.base) -- +(0,-1.7) node{$\dots$};
   \path (c4.base) -- +(-3/8-0.3,-2) coordinate (c4braceleft);
   \path (c4.base) -- +( 3/8+0.3,-2) coordinate (c4braceright);
   \draw[decorate,decoration={mirror,brace,raise=2pt}]
      (c4braceleft) -- node[below=5pt]{$10$} (c4braceright);
\end{tikzpicture}
\end{center}
\caption{Generating tree for $3$-nonnesting open permutation diagrams. The first few levels coincide with those of the tree for all open permutation diagrams.}
\label{fig:permtree}
\end{figure}

\subsubsection{Functional equation}
\label{subsub:funeq}
We translate the generating tree from Theorem~\ref{thm:permgenscheme} into a functional equation for the multivariate
generating function~$F(u,v,w)= \sum F_{h,r,s}(n) u^h v^r w^s z^n$ where~$F_{h,r,s}(n)$ is the number of open permutation arc diagrams at
level $n$ with label $[h, r, s]$. The coefficient~$F_{0,0,0}(n)$
is the number of $3$-nonnesting permutations of $\{1,2,\dots,n\}$.

The translation process to derive a functional equation is analogous to the set partition case. We consider the
five types of vertices in turn to analyze their contribution. The
form of the functional equation is
\[
F(u,v,w) = 1+z(\Psi_1 +\Psi_2+\Psi_3+\Psi_4+\Psi_5),
\]
where $\Psi_i$ is the contribution for adding a vertex of type $(i)$ for $1\le i\le 5$, which we compute next.
\begin{enumerate}
\item Fixed point. Note that case $(1)$ in the succession rule can alternatively be included by extending the range of $i$ in case $(3)$ to include $h$. Thus, it is simpler to compute $\Psi_1+\Psi_3$ in item $(3)$ below.
\item Opener.  $\Psi_2=uF(u,v,w)$.
\item Upper transitory and fixed point. $\Psi_1+\Psi_3=\frac{F(u,v,w) - vF(uv,1,w)}{1-v}
     +\frac{F(u,v,w) - F(u,0,w) }{v}$, found using the formula for a finite geometric sum in the expressions below:
\begin{alignat*}{2}
   & \sum_{h,s,n}F_{h,0,n}(n) u^h (1+v+v^2+ \dots + v^h) w^s z^n
   & \qquad &\text{if $r=0$},
\\
   & \sum_{h,s,n}F_{h,r,s}(n) u^h(v^{r-1}+ v^r +\dots + v^h) w^s z^n
   &\qquad &\text{if $0<r\le i$}.
\end{alignat*}

\item Lower transitory. $\Psi_4=\frac{F(u,v,w) - F(uw,v,1)}{1-w}
     +\frac{F(u,v,w) - F(u,v,0) }{w}$.
 \item Closer.  The addition of a closer to a diagram with label $[h,r,s]$ contributes
\begin{equation*}
\sum_{h, n} F_{h,r,s}(n)u^{h-1} (v^{\max\{r-1,0\}}+ \dots+ v^{h-1}) ( w^{\max\{s-1,0\}} +  \dots + w^{h-1})z^n,
\end{equation*}
which can be simplified using finite geometric sum formulas, and separating the case when $r=0$ or $s=0$:
 \begin{multline*}
\Psi_5= \frac{F(u,v,w)-F(uv,1,w)-F(uw,v,1)+F(uvw,1,1)}{u(1-v)(1-w)}
\\
    + \frac{F(u,v,w)-F(u,0,w)-F(uw,v,1)+F(uw, 0,1)}{uv(1-w)}
\\
    + \frac{F(u,v,w)-F(u,v,0)-F(uv,1,w)+F(uv, 1, 0)}{uw(1-v)}
\\
    + \frac{F(u,v,w)-F(u,0, w)-F(u,v,0)+F(u,0,0)}{uvw}.
\end{multline*}
\end{enumerate}
Adding all five contributions, we get the following corollary.
\begin{corollary}\label{cor:genfun3nonperm}
The generating function for $3$-nonnesting open permutation diagrams, denoted
\[
   F(u,v,w)=F(u,v,w;z)
   = \sum_{h,r,s,n}F_{h,r,s}(n) u^h v^r w^s z^n,
\]
where $F_{h,r,s}(n)$ is the number of
diagrams of size~$n$ with label $[h, r, s]$, satisfies the
functional equation
\begin{multline*}
F(u,v,w)=1+ z\biggl( uF(u,v,w)\\
+ \frac{F(u,v,w)-v(uv, 1,w)}{1-v} + \frac{F(u,v,w)-F(u,0,w)}{v} + \frac{F(u,v,w)-F(uw,v,1)}{1-w}\\
+\frac{F(u,v,w)-F(u,v,0)}{w} +
   \frac{F(u,v,w)-F(uw,1,w)-F(uw,v,1)+F(uvw,1,1)}{u(1-v)(1-w)}
\\
    + \frac{F(u,v,w)-F(u,0,w)-F(uw,v,1)+F(uw, 0,1)}{uv(1-w)}
\\
    + \frac{F(u,v,w)-F(u,v,0)-F(uv,1,w)+F(uv, 1, 0)}{uw(1-v)}
\\
    + \frac{F(u,v,w)-F(u,0, w)-F(u,v,0)+F(u,0,0)}{uvw}
   \biggr).
\end{multline*}
\end{corollary}

We have found this equation useful to generate terms in the sequence, but we have
been unable to solve it, or to find an explicit expression for
$F_{0,0,0}(n)$, the number of $3$-nonnesting permutations.

\subsection{The general case: $k+1$-nonnesting permutations}
The construction in Section~\ref{sec:3perm} can easily be generalized to this case.
Recall that to each $k+1$-nonnesting open permutation diagram, we assign a label
$[h;\r;\s]=[h;r_1,r_2,\dots,r_{k-1};s_1,s_2,\dots,s_{k-1}]$.
To describe the succession rule of the corresponding generating tree,
we think of $[h,\r]$ as the label of the upper set partition, where we consider enhanced nestings (refer to Theorem \ref{thm:kepart}), and of $[h,\s]$
as the label of the lower set partition, where we consider usual nestings (see Theorem \ref{thm:kpart}).
We use~$\r-\vone$ as a shorthand for $r_1-1,r_2-1,\dots,r_{k-1}-1$,
and similarly for~$\s-\vone$. When the parameters $r_0$ and $s_0$ are
used below in $(3b), (4b),$ etc., they are defined to be equal to $h$.

\begin{theorem}\label{thm:kperm}
  Let~$\Sigma^{(k)}$ be the set of $k+1$-nonnesting open permutation diagrams. To each diagram~$\sigma$
  associate the label $\ell(\sigma)=[h;\r;\s]=[h;r_1,r_2,\dots,r_{k-1};s_1,s_2,\dots,s_{k-1}]$,
  where $2h$ is the number of semi-arcs, and $r_i$ (resp. $s_i$) is
  the number of open upper (resp. lower) semi-arcs of enhanced nesting index (resp. nesting index)
  greater than or equal to~$i$. Then the number of diagrams in $\Sigma^{(k)}$ of
  size $n$ is the number of nodes at level $n$ in the generating tree
  with root label $[0;\mathbf{0};\mathbf{0}]$, and succession rule given by
\begin{multline*}
[h;\r;\s]\longrightarrow\\
\quad\begin{array}{llr}
[h;h,r_2,\dots,r_{k-1};\s], &\text{\small (1)}\\[2mm]
[h+1;\r;\s], &\text{\small (2)} \\[2mm]
[h;\r-\vone;\s], \quad \text{if $r_{k-1}\ge1$}, &\text{\small (3a)}\\[2mm]
[h;r_1-\vone,\dots,r_{j-1}-1,i,r_{j+1},\dots,r_{k-1};\s], \quad
   \text{for $1\le j\le k-1$ and $r_j\le i\le r_{j-1}-1$}, &\text{\small (3b)}\\[2mm]
[h;\r;\s-\vone], \quad\text{if $s_{k-1}\ge1$}, &\text{\small (4a)}\\[2mm]
[h;\r;s_1-\vone,\dots,s_{\jmath-1}-1,\imath,s_{\jmath+1},\dots,s_{k-1}],
\quad\text{for $1\le \jmath\le k-1$ and $s_{\jmath}\le \imath\le s_{\jmath-1}-1$},
&\text{\small(4b)}\\[2mm]
[h-1;\r-\vone;\s-\vone],
\quad \text{if $r_{k-1}\ge1$ and $s_{k-1}\ge1$}, &\text{\small (5a)}\\[2mm]
[h-1;\r-\vone;s_1-1,\dots,s_{\jmath-1}-1,\imath,s_{\jmath+1},\dots,s_{k-1}],\\
\multicolumn{1}{r}{\text{if $r_{k-1}\ge1$, for $1\le \jmath\le k-1$ and $s_{\jmath}\le \imath\le s_{\jmath-1}-1$},}&\text{\small (5b)}\\[2mm]
[h-1;r_1-1,\dots,r_{j-1}-1,i,r_{j+1},\dots,r_{k-1};\s-\vone],\\
\multicolumn{1}{r}{\text{if $s_{k-1}\ge1$, for $1\le j\le k-1$ and $r_j\le i\le r_{j-1}-1$},}&\text{\small (5c)}\\[2mm]
[h-1;r_1-1,\dots,r_{j-1}-1,i,r_{j+1},\dots,r_{k-1};s_1-1,\dots,s_{\jmath-1}-1,\imath,s_{\jmath+1},\dots,s_{k-1}],\\
\multicolumn{1}{r}{\text{for $1\le \jmath\le k-1$ and $s_{\jmath}\le \imath\le s_{\jmath-1}-1$, and for $1\le j\le k-1$ and $r_j\le i\le r_{j-1}-1$.}}&\text{\small (5d)}\\[2mm]
\end{array}
\end{multline*}
\end{theorem}
Note that for $k=2$, the generating tree in Theorem~\ref{thm:kperm} agrees with the generating tree defined in Thoerem \ref{thm:permgenscheme} for $\Sigma^{(2)}$.

\begin{proof} The labels correspond to the addition of the following vertices to a diagram $\sigma$:
\begin{itemize}
\item[$(1)$] a fixed point  (as in $(1)$ of Theorem \ref{thm:kepart});
\item[$(2)$] an opener (as in $(2)$ of Theorem \ref{thm:kepart} or $(2)$ of Theorem $\ref{thm:kpart}$);
\item[$(3a)$] an upper transitory closing the top semi-arc, if $\sigma$ has a future enhanced upper $k-1$-nesting (as in $(5)$ of Theorem \ref{thm:kepart});
\item[$(3b)$] an upper transitory (as in $(3)$ of Theorem \ref{thm:kepart});
\item[$(4a)$] a lower transitory closing the bottom semi-arc, if $\sigma$ has a future lower $k-1$-nesting ($(5)$ in Theorem \ref{thm:kpart});
\item[$(4b)$] a lower transitory (as in $(3)$ of Theorem \ref{thm:kpart});
\item[$(5a)$] a closer that closes both the top and the bottom semi-arcs, if $\sigma$ has both a future enhanced upper $k-1$-nesting and a future lower $k-1$-nesting;
\item[$(5b)$] a closer that closes the top semi-arc and a lower semi-arc that is not the bottom one, if $\sigma$ has a future enhanced upper $k-1$-nesting;
\item[$(5c)$]  a closer that closes the bottom semi-arc and an upper semi-arc that is not the top one,
if $\sigma$ has a future lower $k-1$-nesting;
\item[$(5d)$] a closer that closes an upper and a lower semi-arc, neither of which is an outermost one.
\end{itemize}
\end{proof}

The generating function for $(k+1)$-nonnesting open permutation diagrams, denoted by \[
F(u; v_1, v_2, \dots, v_{k-1}; w_1, w_2, \dots, w_{k-1}; z) = F(u, \mathbf{v}, \mathbf{w})
 = \sum_{h, \mathbf{r}, \mathbf{s}, n} F_{h,\mathbf{r},\mathbf{s}} (n) 
 			u^h \mathbf{v}^{\mathbf{r}} \mathbf{w}^{\mathbf{s}} z^n,
\]
 where $F_{h,\mathbf{r},\mathbf{s}} (n)$ is the number of diagrams of size $n$ with label $[ h; \mathbf{r}; \mathbf{s}]$ satisfying the functional equation:
\[
F(u, \mathbf{v}, \mathbf{w}) = 1 + z(\Phi_1 + \Phi_2 + \Phi_3 + \Phi_4 + \Phi_5)
\]
such that each $\Phi_i$ is the contribution for adding a vertex of type $(i)$ in Theorem~\ref{thm:kperm}. We compute each $\Phi_i$ next, following the development of the functional equation for $3$-nonnesting open permutation diagrams in Section~\ref{subsub:funeq}.

\begin{enumerate}
\item Fixed point. By extending the range of $i$ in case $(3)$ for upper transitories, the case $(1)$ in the succession rule  can alternatively be included. Thus it is simpler to compute $\Phi_1 + \Phi_3$ in item $(3)$ below.
\item Opener. $\Phi_2 = uF$.
\item Upper transitory and fixed point:
\[
\Phi_1 + \Phi_3 =  \frac{1}{v_1\dots v_{k-1}}\left(F-F|_{v_{k-1}=0}\right) + \sum_{j=1}^{k-1}  \frac{1}{v_1 \dots v_{j-1}(1-v_j)} 
\left(F-F|_{v_j=1,v_{j-1}=v_{j-1}v_j}\right) 
+F|_{v_1=1,u=uv_1}
\]
\item Lower transitory:
\[
\Phi_4 =\frac{1}{w_1\dots w_{k-1}}\left(F-F|_{w_{k-1}=0}\right) + \sum_{\jmath=1}^{k-1} \frac{1}{w_1\dots w_{\jmath-1}(1-w_{\jmath})}\left(F-F|_{w_{\jmath}=1,w_{\jmath-1}=w_{\jmath-1}w_{\jmath}}\right)
\]
\item Closer: 
\begin{multline*}
\Phi_5 = +\frac{1}{uv_1\dots v_{k-1}w_1\dots w_{k-1}}\left(F-F|_{v_{k-1}=0}-F|_{w_{k-1}=0}+F|_{v_{k-1}=w_{k-1}=0}\right)\\
+\frac{1}{uv_1\dots v_{k-1}}\sum_{\jmath=1}^{k-1}\frac{1}{w_1\dots w_{\jmath-1}(1-w_{\jmath})}\left(F-F|_{v_{k-1}=0}-F|_{w_{\jmath}=1,w_{\jmath-1}=w_{\jmath-1}w_{\jmath}}+F|_{v_{k-1}=0,w_{\jmath}=1,w_{\jmath-1}=w_{\jmath-1}w_{\jmath}}\right)\\
+\frac{1}{uw_1\dots w_{k-1}}\sum_{j=1}^{k-1}\frac{1}{v_1\dots v_{j-1}(1-v_{j})}\left(F-F|_{w_{k-1}=0}-F|_{v_{j}=1,v_{j-1}=v_{j-1}v_{j}}+F|_{w_{k-1}=0,v_{j}=1,v_{j-1}=v_{j-1}v_{j}}\right)\\
+\frac{1}{u}\sum_{j=1}^{k-1}\sum_{\jmath=1}^{k-1} \Biggl(
\frac{1}{v_1\dots v_{j-1}(1-v_{j})w_1\dots w_{\jmath-1}(1-w_{\jmath})} \\
\times \left(F-F|_{v_{j}=1,v_{j-1}=v_{j-1}v_{j}}
-F|_{w_{\jmath}=1,w_{\jmath-1}=w_{\jmath-1}w_{\jmath}}
+F|_{v_{j}=1,v_{j-1}=v_{j-1}v_{j},w_{\jmath}=1,
w_{\jmath-1}=w_{\jmath-1}w_{\jmath}}\right) \Biggr)
\end{multline*}

\end{enumerate}

Note that when $F$ is evaluated at $u=0$, namely disregarding diagrams with openers, we obtain a function of $z$ only. Thus $F|_{u=0}$ is the generating function for permutations avoiding $(k+1)$-nestings. Using the functional equation, we computed the first few terms for $(k+1)$-nonesting permutations as shown in Table~\ref{tab:k+1perm}.

\subsection{Enumerative data}
\begin{table}[htb]\small
\begin{tabular}{ccl}
$k+1$ & OEIS & Initial terms \\
3& A193938 & \bt{l} \textcolor{gray}{1, 2, 6, 24,} 118, 675, 4333, 30464, 230615, 1856336, 15738672, 139509303, 1285276242,\\
12248071935,  120255584181, 1212503440774, 12519867688928, 132079067871313 \et \\
4& A193935 & \bt{l} \textcolor{gray}{1, 2, 6, 24, 120, 720,} 5034, 40087, 356942, 3500551, 37343168, 428886219, 5257753614,\\
 68306562647, 934747457369, 13404687958473, 200554264435218,  3118638648191005 \et \\
5& A193936 & \bt{l} \textcolor{gray}{1, 2, 6, 24, 120, 720, 5040, 40320,} 362856, 3627385, 39864333, 477407104, 6183182389,\\
86033729930,  1278515941177,  20185987771091 \et \\
6& A193937 & \bt{l} \textcolor{gray}{1, 2, 6, 24, 120, 720, 5040, 40320, 362880, 3628800,} 39916680, 478991641, 6226516930,\\
 87157924751, 1306945300264 \et
\end{tabular}
\smallskip
\caption{Counting sequences for $k+1$-nonnesting permutations. Terms that coincide with $n!$ are gray.}
\label{tab:k+1perm}
\end{table}

In 2007, Corteel \cite{Cort07} gave a bijection between weighted
bicoloured Motzkin paths and permutations with $k$ weak exceedences,
$l$ crossings and $m$ nestings. Using continued fractions, she then
determined a generating function for such permutations. Note, however,
that in \cite{Cort07} permutations are parameterized by the total number
of crossings and nestings (that is, $2$-crossings and $2$-nestings),
instead of by sets of $k$ arcs mutually crossing or nesting.

We used the {\it gfun} package of {\sc Maple} (version 3.53) to try to
fit the counting sequences for $k$-nonnesting permutations (with $3\le
k\le 6$) into a differential equation using 80 terms, with no
success. Thus, we make the following conjecture.
\begin{conjecture}\label{conj:notDfinite}
The ordinary generating function for $k$-nonnesting permutations is
not D-finite for any $k>2$.
\end{conjecture}

\section{Perspectives}
\label{sec:perspectives}

We conclude by mentioning some possible future directions of
research extending our work.
Among the conjectures that we have presented, Conjecture~\ref{conj:baxter} seems
to be the most accessible. A more challenging goal would be to find some general
techniques for solving the functional equations that we encounter.

Our work also opens some new options for addressing
Conjecture~\ref{conj:nonDfin}. By studying the refinement of
$k$-noncrossing set partitions by nesting number, an
interesting picture of the singularities of the generating function emerges. This might
be useful in showing that it is not D-finite.

Having found functional equations for the generating functions of
partitions and permutations avoiding $k$-nestings (which are
equinumerous to those avoiding $k$-crossings), a natural extension
would be to describe the distribution of the number of $k$-crossings
and $k$-nestings on partitions and permutations.  More generally,
given $i_1,i_2,\dots,i_r,j_1,j_2,\dots,j_r$, what is the number of
partitions (resp. permutations) of size $n$ with $i_k$ $k$-crossings
and $i_\ell$ $\ell$-nestings for $1\le k,\ell \le r$? Is this number
the same if the words ``crossings" and ``nestings" are switched?  A
reasonable first step towards describing this distribution would be to
study the number of $\ell$-nestings (for $\ell<k$) in $k$-nonnesting
partitions (resp. permutations). Perhaps probabilistic arguments could
be used to obtain information about the expected number of
$\ell$-nestings.
In a similar direction, it may be worth investigating whether a notion
of future $k$-crossings can be used to translate the $k$-noncrossing
condition on open partition diagrams into certain restrictions on some
appropriately defined labels, from where a generating tree for
$k$-noncrossing partitions could be constructed. This might allow one
to impose noncrossing and nonnesting restrictions simultaneously.

Beyond partitions and permutations, our setup is well suited to related combinatorial classes.
Arc diagrams can be used to express the intramolecular interactions of RNA
molecules, where nucleotides (vertices) can form hydrogen bonds (arcs)
which stabilize the structure. In these molecules, it is unlikely for
many mutually crossing hydrogen bonds to be formed. Such configurations correspond to
$k$-crossings.  Furthermore, it is highly
improbable for closely neighbouring nucleotides to form hydrogen bonds,
and it would be interesting to use our open diagram construction to
study $k$-nestings in a such restricted classes.  Similarly,
tangled diagrams are a combinatorial class consisting of arc diagrams
in which each vertex may have $0$, $1$, or $2$ incident arcs. These
can be used to represent all intramolecular bonds in RNA
\cite{ChQiRe08}.  Our method of generating trees for open diagrams can
exhaustively generate $3$-nonnesting tangled diagrams without
inflating vertices of degree $2$ or relying on enumerative results on
$k$-nonnesting matchings.

In addition, there are combinatorial objects which may seem unrelated at first, but which possess a similar recursive structure that can be exploited for exhaustive
generation. This is the case of Skolem sequences, which have been recently studied in~\cite{BuYe13} from this perspective. However, a label permitting the translation
of the generating tree into generating function equations is still not known in this case.

Finally, it is worth mentioning that the generating tree construction in all the combinatorial classes described above is a natural starting point for random generation
purposes.

\section*{Acknowledgements}
We are very grateful to Mogens Lemvig Hansen for his technical
assistance. We are also grateful to anonymous referees for their comments.

\bibliographystyle{plain}



\section{Appendix: Maple code}

\subsection{Maple code for generating sequences of $k$-nonnesting partitions
(including enhanced nestings).}
\begin{lstlisting}
RULE:=proc(label) option remember; #nops(label)=2(k-2)+1
local out, s, ss,i,j,k;
k:= nops(label);
s:= label; ss:= s - [1$k];

out:=
      #1. fixed point
      [s[1], s[1], op(s[3..k])],

      #2. opener
      [s[1]+1, op(s[2..k])];

      #3. transitory - top arc
      if s[k]>0 then
       out:= out, [s[1], op(ss[2..k])];
      fi;

      #4. transitory - other arcs
      for j from 1 to k-1 do
         for i from s[j+1] to s[j]-1 do
           out:= out, [s[1], op(ss[2..j]), i, op(s[j+2..k])];
         od;
      od;

      #5. closer - top arc
      if s[k]>0 then
       out:= out, [s[1]-1, op(ss[2..k])];
      fi;

      #6. closer - other arcs
      for j from 1 to k-1 do
         for i from s[j+1] to s[j]-1 do
           out:= out, [s[1]-1, op(ss[2..j]), i, op(s[j+2..k])];
         od;
      od;


return [out];
end proc:
\end{lstlisting}

For example \lstinline!RULE([1,0,0]);! yields
\[
                [[1, 1, 0], [2, 0, 0], [1, 0, 0], [0, 0, 0]]
\]

\begin{lstlisting}
termtolabeltoterm:= proc (term, N)
 option remember;local out;
   out:=RULE([seq(degree(term, x[i]), i=1..N)]);
subs(seq(x[i]=1, i=1..N), term)*add(mul(x[i]^out[j][i], i=1..N), j=1..nops(out))
end proc:
\end{lstlisting}
\begin{lstlisting}
nextlevel:=proc(l, N) option remember;
local i,out, L;
  out:=0;
  L:= convert(l, list);
  out:= add(termtolabeltoterm(L[i], N), i=1..nops(L));
  return out;
end:
\end{lstlisting}
\begin{lstlisting}
level:=proc(n, K) option remember;
    if n=0 then return 1
    else return nextlevel(level(n-1, K), K); fi;
end:
\end{lstlisting}
For example,
\begin{lstlisting}
seq(subs(seq(x[i]=0, i=1..5),level(n, 2)),n=0..16);
\end{lstlisting}

\subsection{Maple code for generating sequences of k-nonnesting permutations}
\begin{lstlisting}
RULE:=proc(label) option remember; #nops(label)=2(k-2)+1
local out, h,r, s, rr,ss,i,j, ii, jj,k;
#r is for top, s is for bottom

k:= (nops(label)-1)/2+1;
h:= label[1];
r:= label[2..k];       rr:= r - [1$k-1];
s:= label[k+1..nops(label)]; ss:= s - [1$k-1];

out:=
      #1. fixed point
      [h, h, op(r[2..k-1]), op(s)],

      #2. opener
      [h+1, op(r), op(s)];

      #3. upper transitory - top arc
      if r[k-1]>0 then
       out:= out, [h, op(rr), op(s)];
      fi;

      #4. upper transitory - other arcs
      for i from r[1] to h-1 do
           out:= out, [h, i, op(r[2..k-1]), op(s)];
      od;
      for j from 2 to k-1 do
         for i from r[j] to r[j-1]-1 do
           out:= out, [h, op(rr[1..j-1]), i, op(r[j+1..k-1]), op(s)];
         od;
      od;
      #5. upper transitory - top arc
      if s[k-1]>0 then
       out:= out, [h, op(r), op(ss)];

      fi;


      #6. lower transitory - other arcs
      for i from s[1] to h-1 do
           out:= out, [h, op(r), i, op(s[2..k-1])];
      od;
      for j from 2 to k-1 do
         for i from s[j] to s[j-1]-1 do
           out:= out, [h, op(r), op(ss[1..j-1]), i, op(s[j+1..k-1])];
         od;
      od;

      #7. closer - top and bottom arcs
      if r[k-1]>0 and s[k-1]>0 then
        out:= out, [h-1, op(rr), op(ss)];
      fi;


      #8 closer top arc + bottom others
      if r[k-1]>0 then
        for i from s[1] to h-1 do
           out:= out, [h-1, op(rr), i, op(s[2..k-1])];
         od;

        for j from 2 to k-1 do
         for i from s[j] to s[j-1]-1 do
           out:= out, [h-1, op(rr), op(ss[1..j-1]), i, op(s[j+1..k-1])];
         od;
        od;
      fi;

      #9 closer top others + bottom arc
      if s[k-1]>0 then
        #j=1
        for i from r[1] to h-1 do
           out:= out, [h-1, i, op(r[2..k-1]), op(ss)];
         od;
        #j=2
        for j from 2 to k-1 do
         for i from r[j] to r[j-1]-1 do
           out:= out, [h-1, op(rr[1..j-1]), i, op(r[j+1..k-1]), op(ss)];
         od;
        od;
      fi;


      #10 closer: others top + bottom
      #j=1
      for i from s[1] to h-1 do
        #jj=1
        for ii from r[1] to h-1 do
           out:= out, [h-1, ii,  op(r[2..k-1]),
                           i,  op(s[2..k-1])];
        od;

        #jj>1
        for jj from 2 to k-1 do
          for ii from r[jj] to r[jj-1]-1 do
           out:= out, [h-1, op(rr[1..jj-1]), ii, op(r[jj+1..k-1]),
                          i,  op(s[2..k-1])];

          od;
        od;
      od;

      #j>1
      for j from 2 to k-1 do
         for i from s[j] to s[j-1]-1 do
           #jj=1
           for ii from r[1] to h-1 do
           out:= out, [h-1,  ii, op(r[2..k-1]),
                          op(ss[1..j-1]),  i,  op(s[j+1..k-1])];

           od;

        #jj>1
        for jj from 2 to k-1 do
              for ii from r[jj] to r[jj-1]-1 do
              out:= out, [h-1, op(rr[1..jj-1]), ii, op(r[jj+1..k-1]),
                          op(ss[1..j-1]),  i,  op(s[j+1..k-1])];

         od;od;
      od;od;

return [out];
end proc:
\end{lstlisting}
\begin{lstlisting}
termtolabeltoterm:= proc (term, N)
 option remember;local out;
   out:=RULE([seq(degree(term, x[i]), i=1..N)]);
subs(seq(x[i]=1, i=1..N), term)*add(mul(x[i]^out[j][i], i=1..N), j=1..nops(out))
end proc:
\end{lstlisting}
\begin{lstlisting}
nextlevel:=proc(l, N) option remember;
local i,out, L;
  out:=0;
  L:= convert(l, list);
  out:= add(termtolabeltoterm(L[i], N), i=1..nops(L));
  return out;
end:
\end{lstlisting}
\begin{lstlisting}
level:=proc(n, K) option remember;
    if n=0 then return 1
    else return nextlevel(level(n-1, K),2*(K-1)-1); fi;
end:
\end{lstlisting}
For example,
\begin{lstlisting}
k:=5; # choose your desired k
seq(subs(seq(x[i]=0, i=1..5),level(n, k)),n=0..16);
\end{lstlisting}

\end{document}